\documentclass[12pt,letterpaper]{amsart}
\usepackage{fullpage}
\usepackage{hyperref}
\usepackage{amssymb}
\usepackage{enumerate}
\usepackage{amsmath,bm}\usepackage{color}
\usepackage{amsthm}
\usepackage{amssymb}
\usepackage{mathtools}
\usepackage{color}
\usepackage{float}
\usepackage{caption}
\usepackage{tikz,tikz-cd}
\usepackage[utf8]{inputenc}
\usepackage[english]{babel}
\usepackage{amsmath,calligra,mathrsfs}
\usepackage[capitalize,noabbrev]{cleveref}

\usepackage[backend=bibtex,sorting=anyt]{biblatex}
\bibliography{../References}

\usetikzlibrary{matrix}

\newtheorem{theorem}{Theorem}[section]

\newtheorem{corollary}[theorem]{Corollary}
\newtheorem{proposition}[theorem]{Proposition}

\theoremstyle{definition}
\newtheorem{definition}[theorem]{Definition}

\theoremstyle{remark}
\newtheorem{remark}[theorem]{Remark}
\theoremstyle{lemma}
\newtheorem{lemma}[theorem]{Lemma}
\theoremstyle{example}
\newtheorem{example}[theorem]{Example}

\newcommand{\scHom}{\mathscr{H}\text{\kern -5pt {\calligra\large om}}\,}
\newcommand{\shtensor}{\otimes_{\mathbf{sh}}}
\newcommand{\grtensor}{\otimes_{\mathbf{gr}}}

\newcommand{\R}{\mathbb{R}}

\DeclareMathOperator{\Hom}{Hom}

\DeclareMathOperator*{\colim}{colim}

\newcommand{\uHom}{\underline{\Hom}}
\newcommand{\Homk}{\Hom_{\mathbf{k}}}
\newcommand{\cat}{\mathbf}
\newcommand{\Open}{\cat{Open}}

\newcommand{\shRMod}[1]{{#1}\text{-}\mathbf{Mod}}
\newcommand{\shModR}[1]{\mathbf{Mod}\text{-}{#1}}

\definecolor{ffqqqq}{rgb}{1,0,0}
\definecolor{NavyBlue}{rgb}{0.0,0.0,0.5}

\title{Convolution of Persistence Modules}
\author{Nikola Mili\'cevi\'c}

\tikzset{%
    symbol/.style={%
        draw=none,
        every to/.append style={%
            edge node={node [sloped, allow upside down, auto=false]{$#1$}}}
    }
}

\begin{document}

\maketitle


\begin{abstract}
We conduct a study of real-valued multi-parameter persistence modules as sheaves and cosheaves. Using the recent work on the homological algebra for persistence modules, we define two different convolution operations between derived complexes of persistence modules. We show that one of these operations is canonically isomorphic to the derived tensor product of graded modules. We give formulas for computing convolutions between single-parameter interval decomposable modules. Our convolution operations are analogous to the convolution of derived complexes of sheaves on $\mathbb{R}^n$ introduced by Schapira and Kashiwara. In our setting, $\mathbb{R}^n$ has a non-standard topology. We show our convolution operation satisfies analogous properties to the convolution of sheaves on $\mathbb{R}^n$ with the standard topology. We define a convolution distance for derived complexes of persistence modules and show that it extends the classical interleaving distance. We also prove stability results from the sheaf and cosheaf points of view. 
\end{abstract}

\section{Introduction}
Sheaves and cosheaves have found many applications in data science problems of the local-to-global character \cite{ghrist2011applications,hansen2019distributed,hansen2020opinion,robinson2013understanding,robinson2014topological,robinson2017sheaves}. A common perspective in applications is to study sheaves and cosheaves on partially ordered sets often valued in vector spaces over a field $\mathbf{k}$. The thesis work of Curry \cite{curry2013sheaves} showed that functors from a partially ordered set into a ``nice" category are equivalent to sheaves and cosheaves on open and closed sets of the Alexandrov topology \cite{alexandrov1998combinatorial} on the partially ordered set, respectively, valued in said category. An example of this are cellular sheaves and cosheaves \cite{curry2013sheaves,curry2018dualities}. On an arbitrary topological space sheaf cohomology is well defined and studied in the derived setting for any sheaf. On the other hand, cosheaf homology is only defined for constant or locally constant cosheaves. However, on finite partially ordered sets one can construct a rich sheaf cohomology and cosheaf homology theory in the framework of derived functors for any sheaf and cosheaf  \cite{sanchez2020homology}. One can even study entropy and information theory from this point of view \cite{baudot2015homological}.  
  
A particular area of interest is topological data analysis, where sheaves have found applications \cite{de2016categorified,brown2017sheaf,
yoon2020persistence,kashiwara2018persistent,
berkouk2018derived,bubenik2019homological}. In topological data analysis, data is often encoded as a diagram of topological spaces. An appropriate homology functor with field coefficients is applied to obtain a diagram of vector spaces. When data is parametrized by a number of real variables, this diagram of vector spaces is indexed by $\R^n$ and is called a (real-valued multi-parameter) persistence module.   
  
Even though persistence modules have been studied extensively, it is a relatively recent endeavor to study them using homological algebra techniques \cite{miller2020homological,bubenik2019homological,
botnan2020relative,berkouk2019stable,gakhar2019k,
carlsson2020persistent,govc2020persistent}. This work reconsiders classical results such as the defininition of interleaving distance and its stability theorems from the homological algebra perspective. We consider persistence modules from the sheaf and cosheaf points of view. Motivated by the work done by Schapira and Kashiwara in \cite{kashiwara2018persistent}, where a convolution operation derived complexes of sheaves of vector spaces is used to define a distance, we define two different operations for derived complexes of persistence modules; the sheaf and cosheaf convolutions. We show these convolution operations satisfy properties analogous to the original definition. We use the convolutions to define two distances for derived complexes of persistence modules, that turn out to be equivalent. Furthermore, when restricted to persistence modules thought of as complexes concentrated in degree $0$, the convolution distance recovers the classical interleaving distance of persistence modules. In particular, if we denote by $d_C$ the convolution distance between $M$ and $N$ and by $d_I$ the  interleaving distance of persistence modules we have the following result:

\begin{theorem}
Let $M$ and $N$ be persistence modules thought of as complexes concentrated in degree $0$, then $d_C(M,N)=d_I(M,N)$.
\end{theorem}

The original motivation for this work, were two bifunctors of persistence modules, $\grtensor$ and $\uHom$, the graded module tensor product and its adjoint internal hom,  whose homological algebra was studied with detail in \cite{bubenik2019homological}. In sheaf theory, the six Grothendieck operations are well known and often a functor of interest is canonically isomorphic to some combination of them \cite{kashiwara1990sheaves}. In this work we show that $\grtensor$ is canonically isomorphic to the composition of the external tensor product, $\boxtimes$, and a direct image functor of cosheaves. Similarly, we show that $\uHom$ is canonically isomorphic to a composition of inverse image functors, the sheaf hom, $\scHom$, and a direct image functors of sheaves. More specifically we have the following two theorems:

\begin{theorem}
Let $M$ and $N$ be bounded derived complexes of persistence modules. We define the cosheaf convolution of $M$ and $N$, $M\bullet^LN:=Ls_{\dagger}(M\boxtimes N)$. There exists a canonical isomorphism
\[M\bullet^LN\cong M\grtensor^LN.\] 
\end{theorem}

\begin{theorem}
Let $M$ and $N$ be bounded derived complexes of persistence modules. We define $R\scHom^*(M,N):=R\pi_{2*}\scHom(\pi_1^{-1}M,s^{-1}N)$.  There exists a canonical isomorphism
\[R\scHom^*(M,N)\cong R\uHom (M,N).\]
\end{theorem}

 We provide explicit formulas for interval modules arising from the persistent homology of sublevel sets of functions, for the sheaf and cosheaf convolutions. In applications, single-parameter persistence modules decompose into direct sums of finitely many such interval modules. Since the sheaf and cosheaf convolutions are additive functors and thus preserve finite direct sums, the general case reduces to that of interval modules.

\begin{proposition}
Let $M=\mathbf{k}[a,b)$ and $N=\mathbf{k}[c,d)$ be interval modules. The cohomology complex of the sheaf convolution $M*^RN$ is given by: 
\begin{itemize}
\item $(M*^RN)^0=\mathbf{k}[\max(a+d,b+c),b+d)$
\item $(M*^RN)^1=\mathbf{k}[a+c,\min(a+d,b+c))$
\item $(M*^RN)^i=0$ for $i\ge 2$.
\end{itemize}
The homology complex of the cosheaf convolution $M\bullet^{L}N$ is given by:
\begin{itemize}
\item $(M\bullet^LN)_0=\mathbf{k}[a+c,\min(a+d,b+c))$
\item $(M\bullet^LN)_1=\mathbf{k}[\max(a+d,b+c),b+d)$
\item $(M\bullet^LN)_i=0$ for $i\ge 2$.
\end{itemize}
\end{proposition}
\noindent

The paper is structured as follows. In \cref{section:background} we recall the necessary background for the rest of the paper. In \cref{section:convolution} we define the sheaf and cosheaf convolution for derived complexes of persistence modules and the induced convolution distance. We prove various properties of these functors and compute examples for single-parameter interval decomposable modules. In \cref{section:stability} we prove stability results for the convolution distance. In \cref{section:concluding_remarks} we summarize the work done in the paper and discuss possible future work. We also give a short background on sheaf theory in the Appendix.

\subsection*{Related work}
A distance on the derived complexes of persistence modules, developed from the homological algebra perspective has been considered recently by Berbouk in \cite{berkouk2019stable}. The convolution distance defined in this paper agrees with the derived interleaving distance defined in \cite{berkouk2019stable}. Berbouk in \cite{berkouk2019stable} relies on the translation functor for his definition, which turns out to be an exact functor. On the other hand, we show that we recover the translation functor by convolving with particular persistence modules. Thus, you can also find Theorem 1.1 in \cite{berkouk2019stable}. Our approach is therefore complementary and examines the homological algebra of different functors that can be used to define the same distance. 
Thanks to our (co)sheaf-theoretic approach we are able to state stability theorems for derived direct images of sheaves and cosheaves of vector spaces on a topological space $X$, relative to a map $f:X\to \mathbb{R}^n$. In \cite{berkouk2019ephemeral}, Berbouk and Petit study various derived interleaving distances on persistence modules, sheaves on $\mathbb{R}^n$ and $\gamma$ sheaves. Berbouk and Petit show isometry theorems between these distances. It is likely that some of our stability theorems can be recovered implicitly from the isometry theorems in \cite{berkouk2019ephemeral}. Our formulation of the functor $R\scHom^{*}$ drew inspiration from the functor $\scHom^{\star}$ introduced as a right adjoint to the convolution operation for derived complexes of sheaves by Tamarkin in \cite{tamarkin2013microlocal} and further studied by Schapira and Guillermou in \cite{guillermou2014microlocal}. However, even though analogous in symbols the underlying definition is different. 

\section{Background}
\label{section:background}
In this section we introduce the necessary background for the rest of the paper. We assume that the reader is familiar with basic definitions and facts about sheaves, cosheaves, homological algebra and derived functors, although some background will also be provided in the paper. For a more detailed exposition, see \cite{kashiwara1990sheaves,Bredon:SheafTheory}. We first introduce notation that will be used throughout the paper. 

\begin{itemize}
\item If $\cat{C}$ and $\cat{D}$ are two categories, we denote by $\cat{D}^{\cat{C}}$ the \emph{functor category} of functors $F:\cat{C}\to \cat{D}$ and natural transformations.
\item If $\mathbf{k}$ is a field, we denote by $\mathbf{Vect}_{\mathbf{k}}$ the category of $\mathbf{k}$-vector spaces and $\mathbf{k}$-linear maps. 
\item If $\cat{C}$ is a complete and cocomplete category, we denote by $\cat{Sh}(X;\cat{C})$ the category of $\cat{C}$-valued sheaves and by $\cat{CoSh}(X;\cat{C})$ the category of $\cat{C}$-valued cosheaves on a topological space $X$, respectively. If $\cat{k}$ is a field, for simplicity we denote by $\cat{Sh}(X;\cat{k})$ and by $\cat{CoSh}(X;\cat{k})$ the categories $\cat{Sh}(X;\cat{Vect}_{\cat{k}})$ and $\cat{CoSh}(X;\cat{Vect}_{\cat{k}})$, respectively.
\item If $(P,\le)$ is a preordered set, we denote by $\cat{P}$ the category whose objects are the elements of $P$ and there is a unique morphism $x\to y$ if and only $x\le y$.
\item If $\cat{A}$ is an abelian category we denote by $C^b(\cat{A})$ its category of bounded complexes, by $K^b(\cat{A})$ its bounded homotopy category and by $D^b(\cat{A})$ its bounded derived category.
\item If $M$ is in $C(\cat{A})$ we denote by $H^{\bullet}M$ ($H_{\bullet}M$) the cohomology (homology) complex of $M$, that is $H^{\bullet}M^n:=H^n(M)$ $(H_{\bullet}M:=H_n(M))$ and all the coboundary (boundary) maps are $0$.
\end{itemize}

\subsection{Convolution of sheaves on euclidean space}
We recall the convolution operation , $\star$, and some of its properties for bounded derived complexes of sheaves of $\mathbf{k}$-vector spaces,  $D^b(\cat{Sh}(V;\mathbf{k}))$, for some field $\mathbf{k}$ on $V=\mathbb{R}^n$. For a more detailed introduction and proofs see \cite{kashiwara2018persistent,
berkouk2018derived,guillermou2014microlocal,
tamarkin2013microlocal}. The purpose of this section is to recall results for complexes of sheaves on $\mathbb{R}^n$ with its standard topology. In later sections we define analogous convolution operations for sheaves and cosheaves on an arbitrary preordered set, $(P,\le,+,0)$, with a compatible abelian group structure with the Alexandrov topology (\cref{def:alexandrov}) and show results that are analogous to the ones presented below. Hence, the definitions and results that follow only serve as a reference in order to make the paper somewhat self-contained.

Let $s:\mathbb{R}^n\times\mathbb{R}^n\to \mathbb{R}^n$ be the addition map $s(x,y):=x+y$ and let $\pi_{i}:\mathbb{R}^n\times \mathbb{R}^n\to \mathbb{R}^n$ for $i=1,2$ denote the canonical projections. 

\begin{definition}
\label{def:convolution_of_constructible_sheaves}
Let $M,N\in D^b(\cat{Sh}(V;\mathbf{k}))$. The convolution $M\star N$ is defined to be
\[M\star N:=Rs_{!}(M\boxtimes N),\]
where $\boxtimes$ denotes the \emph{external tensor product of sheaves} and $Rs_!$ is the \emph{right derived functor of the direct image with proper support} of the addition map $s$. See \cref{section:background_sheaves} for more details. 
\end{definition}

The convolution functor has right adjoint $\scHom^{\star}$, first introduced by Tamarkin in \cite{tamarkin2013microlocal} and then reintroduced by Schapira and Guillermou in \cite{guillermou2014microlocal}. It is the functor in the derived category defined below.

\begin{definition}
\label{def:convolution_adjoint}
Let $F$ and $G$ be in $D^b(\cat{Sh}(V,\cat{k}))$. Define $\scHom^{\star}(-,-):(D^b(\cat{Sh}(V,\cat{k})))^{op}\times D^b(\cat{Sh}(V,\cat{k}))\to D^b(\cat{Sh}(V,\cat{k}))$ by
\[\scHom^{\star}(F,G):=R\pi_{1*}R\scHom(\pi_2^{-1}G,s^{!}F),\]
where $R\scHom$ is the \emph{right derived functor of the sheaf hom functor}, $\pi_2^{-1}$ is the \emph{inverse image functor} of the map $\pi_2$, $R\pi_{1*}$ is the \emph{right derived functor of the direct image functor} of the map $\pi_1$ and $s^{!}$ is the \emph{exceptional inverse image functor} of the addition map $s$. See \cite{kashiwara1990sheaves} for more details.
\end{definition}

Let $Z$ be a closed subset of a topological space $X$ and let $j:Z\to X$ denote the inclusion map. If $F$ is a sheaf of abelian groups on $X$ we can define a new sheaf on $X$, $F_Z:=j_*j^{-1}F$, where $j_*$ and $j^{-1}$ are the classical direct and inverse image functors of sheaves respectively. If $\cat{k}$ is a field it is common to denote by $\cat{k}_X$ the constant sheaf on $X$ that assigns to every open set the field $\cat{k}$. If $Z\subseteq X$ is closed it is common to set $\cat{k}_Z:=j_*j^{-1}\cat{k}_X$.

\begin{definition}
\label{def:convolving_object_constructible_sheaf}
For $\epsilon\ge 0$, let $K_{\epsilon}:=K_{\epsilon}:=\mathbf{k}_{B_{\epsilon}}$ where $B_{\epsilon}:=\{x\in V\,|\, ||x||\le \epsilon\}$ and $\mathbf{k}_{B_{\epsilon}}$ is considered as a complex concentrated in degree $0$ in $D^b(\cat{Sh}(V;\cat{k}))$. For $\epsilon<0$, set $K_{\epsilon}:=\mathbf{k}_{\{x\in V\,|\, ||x||<-\epsilon\}}[n]$, where $n$ is the dimension of $V$ (Recall that if $X$ is a chain complex, $X[n]$ is the shifted complex $X[n]^m:=X^{n+m}$).
\end{definition}

\begin{proposition}
\label{prop:properties_of_convolving_object_constr_sheaf}
Let $\epsilon,\delta\in \mathbb{R}$ and let $F$ be a complex in $D^b(\cat{Sh}(V;\cat{k}))$.
\begin{itemize}
\item[1.] There are natural isomorphisms $(F\star K_{\epsilon})\star K_{\delta}\cong F\star K_{\epsilon+\delta}$ and $F\star K_0\cong F$.
\item[2.] If $\delta\ge  \epsilon$, there is a canonical morphism $K_{\delta}\to K_{\epsilon}$ in $D^b(\mathbf{k}_V)$ inducing a canonical morphism $F\star K_{\delta}\to F\star K_{\epsilon}$. 
\item[3.] The canonical morphism $F\star K_{\delta}\to F\star K_{\epsilon}$ induces an isomorphism $R\Gamma(V,F\star K_{\delta})\to R\Gamma(V,F\star K_{\epsilon})$ and hence an isomorphism  in cohomology.
\end{itemize}
\end{proposition}

\begin{definition}
\label{def:convolution_distance_constructible_sheaf}
Let $F$ and $G$ be in $D^b(\cat{Sh}(V;\cat{k}))$ and let $a\ge 0$. We say $F$ and $G$ are $a$-isomorphic if there are morphisms $f:K_a\star F\to G$ and $g:K_a\star G\to F$ such that the composition $K_{2a}\star F\xrightarrow{K_a\star f}K_a\star G\xrightarrow{g} F$ coincides with the natural morphism $K_{2a}\star F\to F$ and the composition $K_{2a}\star G\xrightarrow{K_a\star g}K_a\star F\xrightarrow{g} G$ coincides with the natural morphism $K_{2a}\star G\to G$. If $F$ and $G$ are $a$-isomorphic, then they are $b$-isomorphic for any $b\ge a$. Thus, one can define the \emph{convolution distance between $F$ and $G$} by:
\[\text{dist}(F,G)=\inf(\{+\infty\}\cup \{a\in \mathbb{R}_{\ge 0}\,|\, \text{F and G are a-isomorphic}\})\]
\end{definition}

Note that $F$ and $G$ are $0$-isomorphic if and only if $F\cong G$.
\begin{proposition}
\label{prop:convolution_is_a_distance_contructible}
The convolution distance is an extended pseudo metric on $D^b(\mathbf{Sh}(V;\mathbf{k}))$, that is for all $F,G,H\in \text{Ob}(D^b(\mathbf{Sh}(V;\mathbf{k})))$ we have:
\begin{itemize}
\item $\emph{dist}(F,G)=\emph{dist}(G,F)$,
\item $\emph{dist}(F,G)\le \emph{dist}(F,H)+\emph{dist}(H,G)$.
\end{itemize}
\end{proposition}

\subsection{Persistence modules}
Here we recall facts about persistence modules that will be necessary for the rest of the paper. Let $(P,\le)$ be a preordered set. A subset $U\subseteq P$ is called an \emph{up-set} if whenever $x\in U$ and $x\le y$ then $y\in U$. A \emph{principal up-set at $x$} is the up-set $U_x$ defined by $U_x:=\{y\,|\, x\le y\}$. A subset $D\subseteq P$ is called a \emph{down-set} if whenever $x\in D$ and $y\le x$ then $y\in D$. A \emph{principal down-set at $x$} is the down-set $D_x$ defined by $D_x:=\{y\,|\, y\le x\}$. Given a preordered set $(P,\le)$ we also have the preordered set $(P,\le^{op})$ where $x\le^{op}y$ if and only of $y\le x$ for all $x,y\in P$. The example we will be dealing the most in this paper is the poset $(\mathbb{R}^n,\le)$ where $x\le y$ if and only if $x_i\le y_i$ for all $1\le i\le n$. 

\begin{definition} 
\label{def:alexandrov}
 Let $(P,\le)$ be a preordered set. Define the \emph{Alexandrov topology} on $P$ to be the topology whose open sets are the up-sets in $P$. Let $\Open(P)$ denote the category whose objects are the open sets in $P$ and whose morphisms are given by inclusions. The \emph{opposite Alexandrov topology} on $P$ is the topology whose open sets are the up-sets with respect to the opposite preorder $\le^{op}$. Equivalently, the open sets of the Alexandrov topology of $(P,\le^{op})$ are the closed sets of the Alexandrov topology of $(P,\le)$.
\end{definition}
 
\begin{lemma}
Let $(P,\leq)$ and $(Q,\leq)$ be preordered sets and consider $P$ and $Q$ together with their corresponding Alexandrov topologies. Let $f: P \to Q$ be a map of sets. Then $f$ is order-preserving  if and only if $f$ is continuous.
\end{lemma}

\begin{example}
Consider $(\mathbb{R},\le)$ with the Alexandrov topology. Then the open sets are $\emptyset$, $\R$, and the intervals $(a,\infty)$ and $[a,\infty)$, where $a \in \R$. Similarly, the open sets of $(\mathbb{R},\le^{op})$ are the $\emptyset$, $\mathbb{R}$ and the intervals  $(-\infty,a)$ and $(-\infty,a]$ where $a\in \mathbb{R}$.
\end{example}


\begin{definition}
\label{def:persistence_modules_as_functors}
A \emph{persistence module} is a functor $M:\cat{P}\to \mathbf{Vect}_{\mathbf{k}}$. A morphims $f:M\to N$  between persistence modules $M$ and $N$ is a natural transformation. 
\end{definition}

Although we defined persistence modules on arbitrary preordered sets, it turns out that when we endow the preordered set with a compatible abelian group structure we get a rich homological algebra theory of persistence modules \cite{miller2020homological,bubenik2019homological}. This is because in this setting persistence modules can be canonically identified with graded modules where the grading is over the preordered set. By a preordered set with a compatible abelian group structure we mean a 4-tuple $(P,\le,+,0)$ where $+$ is an abelian group operation and $0$ is the identity and $a\le b$ implies $a+c\le b+c$ for all $a,b,c\in P$.

\begin{definition}
\label{def:graded_module_tensor_product}
Let $M,N:\cat{P}\to \cat{Vect}_{\cat{k}}$ be persistence modules. Assume that $(P,\le)$ has a compatible abelian group structure. The \emph{graded module tensor product of $M$ and $N$} is the persistence module $M\grtensor N:\cat{P}\to \mathbf{Vect}_{\mathbf{k}}$ defined by $(M\grtensor N)_x:=\colim_{a+b\le x}M_a\otimes_{\mathbf{k}}N_b$. The bifunctor $-\grtensor-$ is right exact is each argument and we can consider its left derived functor, $-\grtensor^L-$. If the functor $M\grtensor -$ or $-\grtensor M$ is exact, we say the persistence module $M$ is \emph{$\grtensor$-flat}. As the name implies, this tensor product is canonically isomorphic with the tensor product of graded modules over a certain $P$ graded ring,  when we think of $M$ and $N$ as graded modules over said ring. It is not important for this paper to identify the graded ring in question as we can give the above functorial definition of $-\grtensor-$.  See for example \cite{bubenik2019homological} for more details involving the graded module point of view.
\end{definition}

\begin{theorem}{\cite[Theorem 4.2.10]{curry2013sheaves}}
\label{theorem:functors_on_posets_are_sheaves_and_cosheaves}
Let $(P,\le )$ be a preordered set and let $\mathbf{C}$ be a complete (respectively cocomplete) category. Then there is an isomorphism of categories $\mathbf{P}^{\mathbf{C}}\cong \mathbf{Sh}(P;\mathbf{C})$ (respectively $\mathbf{P}^{\mathbf{C}}\cong \mathbf{CoSh}(P^{\text{op}};\mathbf{C}))$. 
\end{theorem}

\begin{proof}
We only give the main ideas of the proof as they will be used later in the paper. There are embeddings of pre-ordered sets, $\iota:P \to \Open(P)$ and $j:P\to \Open(P^{op})$ where $\Open(P)$ are the up-sets in $P$ and $\Open({P^{op}})$ are the up-sets in $P^{op}$ or equivalently, down-sets in $P$. The orders on $\Open(P)$ and $\Open(P^{op})$ are given by inclusions of subsets. The embeddings are given by $\iota(x):=U_x$ and $j(x):=D_x$. 

Given a functor $F$ in $\cat{P}^{\cat{C}}$ abusing notation we can define a presheaf, that turns out to be a sheaf, $F$ in $\cat{Sh}(P,\cat{k})$ by a right Kan extension along $\iota$, $F(U):=\text{Ran}_{\iota}F(U):=\lim_{x\in U} F_x$ for every up-set $U\subseteq P$.
Dually, we can define a cosheaf $F$ in $\cat{CoSh}(P^{op},\cat{k})$ by a left Kan extension along $j$, $F(D):=\text{Lan}_jF(D):=\colim_{x\in D} D_x$ for every down-set $D\subseteq P$.

Going the other way, given a sheaf in $\cat{Sh}(P;\cat{C})$ or a sheaf $\cat{CoSh}(P^{op};\cat{C})$ we can define a functor in $\cat{P}^{\cat{C}}$ by considering the sheaf stalks or the cosheaf costalks respectively. These two constructions give us the isomorphisms of categories.

In particular, for every principal down-set and up-set $D_x$ and $U_x$, we have the following equalities, $F(D_x)=F_x=F(U_x)$ and thus these Kan extensions are actual extensions.
\end{proof}

An immediate consequence of \cref{theorem:functors_on_posets_are_sheaves_and_cosheaves}
and the fact that $\mathbf{Vect}_{\mathbf{k}}$ is complete and cocomplete, we have that \cref{def:persistence_modules_as_sheaves,def:persistence_modules_as_cosheaves}
are equivalent to \cref{def:persistence_modules_as_functors}.

\begin{definition}
\label{def:persistence_modules_as_sheaves}
A persistence module $M$ is a sheaf on $(P,\le)$ valued in $\mathbf{Vect}_{\mathbf{k}}$. where the open sets are the up-sets.
\end{definition}

\begin{definition}
\label{def:persistence_modules_as_cosheaves}
A persistence module $M$ is a cosheaf on $(P,\le)$ valued in $\mathbf{Vect}_{\mathbf{k}}$. where the open sets are the down-sets, i.e., the opposite topology of the one in \cref{def:persistence_modules_as_sheaves}. 
\end{definition}

From now on in this paper we assume persistence modules are always functors on a preordered set with a compatible abelian group structure.

Let $M:\cat{P}\to \mathbf{Vect}_{\mathbf{k}}$ be a persistence module. If $a\in P$ we define a persistence module $M(a)$, by $M(a)_x:=M_{x+a}$. A morphism $f:M\to N$ of persistence modules induces an obvious morphism $f(a):M(a)\to N(a)$. If $A\subseteq P$ define $A(a)=\{x\in P\,|\, \exists y\in A, x=y+a\}$. Having the sheaf and cosheaf point of view as in \cref{theorem:functors_on_posets_are_sheaves_and_cosheaves}, it follows that $M(a)(U)=M(U(a))$ and $M(a)(D)=M(D(a))$ for every up-set $U$ and every down-set $D$ in $P$. If $M$ is a complex of persistence modules we can also define $M(a)$ to be the complex whose $n$-term is $M(a)^n:=M^n(a)$. Given a chain map $f:M\to N$ between persistence modules, we get an obvious chain map $f(a):M(a)\to N(a)$.

\begin{definition}
\label{def:internal_hom}
Let $M$ and $N$ be persistence modules. We can construct a new persistence module $\uHom(M,N):\cat{P}\to \mathbf{Vect}_{\mathbf{k}}$ defined by $\uHom(M,N)_x:=\Hom(M,N(x))$. It was shown in \cite[Proposition 4.6]{bubenik2019homological} that there is a canonical isomorphism $\uHom(M,N)_x\cong \lim_{a+b\ge x}\Hom_{\mathbf{k}}(M_{-a},N_b)$.
\end{definition}

The following is perhaps a non-standard definition of the interleaving distance between persistence modules on $\mathbb{R}^n$, but a reader familiar with persistence modules will quickly realize it is equivalent to the standard definition. 

\begin{definition}
\label{def:interleaving_distance_of_persistence_modules}
Let $M,N:\mathbb{R}^n\to \mathbf{Vect}_{\mathbf{k}}$ be persistence modules. Let $\epsilon \in [0,\infty)$ and consider the vector $\bm{\epsilon}\in \mathbb{R}^n$, where $\mathbf{\bm{\epsilon}}_i=\epsilon$ for all $1\le i\le n$. We say $M$ and $N$ are \emph{$\epsilon$-interleaved} if there exists a pair of morphisms of $f:M(-\bm{\epsilon})\to N$ and $g:N(-\bm{\epsilon})\to M$ such that the compositions $M(-2\bm{\epsilon})\xrightarrow{f(-\bm{\epsilon})}N(-\bm{\epsilon})\xrightarrow{g}  M$ and $N(-2\bm{\epsilon})\xrightarrow{g(-\bm{\epsilon})} M(-\bm{\epsilon} )\xrightarrow{f} N$ are equal to the natural transformations whose components are $M_{x-2\bm{\epsilon}\le x}$ and $N_{x-2\bm{\epsilon}\le x}$ for $x\in \mathbb{R}^n$, respectively. If $M$ and $N$ are $\epsilon$-interleaved, then they are $\delta$-isomorphic for any $\delta\ge \epsilon$. Thus, we can define the \emph{interleaving distance} between $M$ and $N$ to be the following extended pseudo-metric
\[d_I(M,N):=\inf(\{+\infty\}\cup\{\epsilon\in \mathbb{R}_{\ge 0}\,|\, \text{M and N are }\epsilon\text{-interleaved}\})\]
\end{definition}

\begin{definition}
\label{def:convex_connected}
Let $(P,\le)$ be a preordered set. A subset $A\subseteq P$ is \emph{convex with respect to} $\le$ if $a\le c\le b$ with $a,b\in A$ implies that $c\in A$. A subset $A$ is \emph{connected with respect to} $\le$ if for any two $a,b\in A$ there exists a sequence $a=p_0\le q_1\ge p_1\le q_2\ge \dots p_n \le q_n=b$ for some $n\in \mathbb{N}$ such that all $p_i,q_i\in A$ for $0\le i\le n$. A connected and convex subset $A\subseteq P$ is called an \emph{interval}. If $A\subseteq P$ is an interval, we will denote by $\mathbf{k}[A]$ the \emph{interval persistence module over $A$}, that is, $\mathbf{k}[A]: \cat{P} \to \mathbf{Vect}_{\mathbf{k}}$ is given by $\mathbf{k}[A]_a=\mathbf{k}$ if $a\in A$ and is $0$ otherwise and all the maps $\mathbf{k}[A]_{a\le b}$, where $a,b \in A$, are identity maps. If $A$ is an interval on the real line $\mathbb{R}$, say $A=[a,b)$, we will write $\mathbf{k}[a,b)$ instead of $\mathbf{k}[[a,b)]$ for brevity. Note that every up-set $U\subseteq \mathbb{R}^n$ and every down-set $D\subseteq \mathbb{R}^n$ is an interval.
\end{definition}

In order to conduct homological algebra computations we need to know if our category of interest has enough projectives and injectives. We are in luck as that is true since $\cat{Vect}_{\cat{k}}^{\cat{P}}$ is a Grothendieck category and furthermore we undestand somewhat which interval modules are injective and which ones are $\grtensor$-flat. See \cite{bubenik2019homological} for more details.

\begin{proposition}
\label{prop:enough_injectives_and_projectives}
The category of persistence modules, $\mathbf{Vect}_{\mathbf{k}}^{\mathbf{R}^n}$, has enough projectives and enough injectives.
\end{proposition}

\begin{definition}
\label{def:lattice}
Let $(P,\le)$ be a poset. Let $a,b\in P$. The \emph{join} of $a$ and $b$ denoted by $a\vee b$ is the smallest $c\in P$ such that $a\le c$ and $b\le c$, if it exists. The \emph{meet} of $a$ and $b$ denoted by $a\wedge b$ is the largest $c\in P$ such that $c\le a$ and $c\le b$, if it exists. A poset where every join exists is called a \emph{join semilattice} and a poset where every meet exists is called a \emph{meet semillatice}. If a poset is both a meet and join semilattice it is called a \emph{lattice}. Note that every upset $U$ and every downset $D$ in a lattice is an interval.
\end{definition}

\begin{example}
\label{example:lattice}
The poset $(\mathbb{R}^n,\le)$ is a lattice.
\end{example}

\begin{proposition}{\cite[Proposition 6.8]{bubenik2019homological}}
\label{prop:classification_into_flats_and_injectives} 
Let $(P,\le)$ be a lattice. Let $D\subseteq P$ be a down-set such that for all $a,b\in D$ the join $a\vee b$ is in $D$. Then the interval module $\mathbf{k}[D]$ is injective. Let $U\subseteq \mathbb{R}^n$ be an up-set  such that for all $a,b\in U$ the meet $a\wedge b$ is in $U$. Then the interval module $\mathbf{k}[U]$ is $\grtensor$-flat.
\end{proposition}

\begin{definition}{\cite[Definitions 1.1.3 and 1.5.2]{sanchez2020homology}}
\label{def:sections_and_cosections}
Let $(P,\le)$ be a preorder and let $F:\mathbf{P}\to \mathbf{Vect}_{\mathbf{k}}$ be a functor.  By \cref{theorem:functors_on_posets_are_sheaves_and_cosheaves}, $F$ is a sheaf on the up-sets of $P$ and a cosheaf on the down-sets of $P$. For a subset $S\subseteq P$ we define
\[\Gamma(S;F):=\lim_{s\in S} F_s,\,\,\,\, L(S;F):=\colim_{s\in S}F_s\]
which we call the \emph{sections} and \emph{cosections} of $F$ on $S$ respectively.
The functors $\Gamma(S;-):\mathbf{Vect}_{\mathbf{k}}^{\mathbf{P}}\to \mathbf{Vect}_{\mathbf{k}}$ and $L(S;-):\mathbf{Vect}_{\mathbf{k}}^{\mathbf{P}}\to \mathbf{Vect}_{\mathbf{k}}$ are left exact and right exact respectively. By \cref{prop:enough_injectives_and_projectives} we can define their derived functors. We label them as $R\Gamma(S;-)$ and $LL(S;-)$ or just $\mathbb{L}(S;-)$, respectively. 
\end{definition}

\subsection{Direct and inverse images}
Here we recall some basic facts about the direct and inverse image functors of sheaves and cosheaves. See for example \cite{curry2013sheaves,kashiwara1990sheaves}.

\begin{definition}{\cite[Definition 2.3.1]{kashiwara1990sheaves}}
\label{def:direct_and_inverse_image_sheaf}
Let $f:Y\to X$ be continuous.
\begin{itemize}
\item[1)] Let $G$ be sheaf on $Y$. The \emph{direct image of $G$ by $f$}, denoted $f_*G$ is the sheaf on $X$ defined by:
\[U\mapsto f_*G(U):=G(f^{-1}(U))\]
for all $U\subseteq X$ open.
\item[2)] Let $F$ be a sheaf on $X$. The inverse image of $F$ by $f$, denoted $f^{-1}F$, is the sheaf on $Y$ associated to the presheaf:
\[V\mapsto \colim\limits_{U}F(U)\]
where $V\subseteq Y$ is open and $U$ ranges through the family of open neighborhoods of $f(V)$ in $X$. 
\end{itemize}
\end{definition}


\begin{definition}
\label{def:direct_image_cosheaf}
Let $f:Y\to X$ be continuous.
Let $G$ be cosheaf on $Y$. The \emph{direct image of $G$ by $f$}, denoted $f_{\dagger}G$ is the cosheaf on $X$ defined by:
\[U\mapsto f_{\dagger}G(U):=G(f^{-1}(U))\]
for all $U\subseteq X$ open.
\end{definition}

We would like to define an inverse image functor for cosheaves in a similar fashian as the inverse image of sheaves. The problem is that in general cosheafifaction does not necessarily exists, however for cosheaves of vector spaces we have the following result.

\begin{theorem}{\cite[Corollary 2.3]{curry2013cosheafification}}
\label{theorem:cosheafification}
The  category  of  cosheaves  of  vector  spaces  is  a  coreflective  subcategory precosheaves of vector spaces, i.e., cosheafification exists.
\end{theorem}

Thus, we have the following defintion.

\begin{definition}
\label{def:inverse_image_cosheaf}
Let $F$ be a cosheaf of vector spaces on $X$. The inverse image of $F$ by $f$, denoted $f^{-1}F$, is the cosheaf on $Y$ associated to the precosheaf:
\[V\mapsto \lim\limits_{U}F(U)\]
where $V\subseteq Y$ is open and $U$ ranges through the family of open neighborhoods of $f(V)$ in $X$. 
\end{definition}

\begin{example}
\label{example:inverse_image_cosheaf_on_a_poset}
If $(X,\le_X)$ and $(Y,\le_Y)$ are preordered sets and let $f:X\to Y$ be continuous with respect to the Alexandrov topologies on $X$ and $Y$. If $F$ is a sheaf of vector spaces on $Y$, then $f^{-1}F_x\cong F_{f(x)}$. If $F$ is a cosheaf of vector spaces on $Y$, then we also have $f^{-1}F_x\cong F_{f(x)}$.
\end{example}

\begin{theorem}{\cite[Theorem 5.3.2]{curry2013sheaves}}
\label{theorem:direct_inverse_image_adjunction_cosheaf}
Let $(X,\le_X)$ and $(Y,\le_Y)$ be preordered sets with the Alexandrov topology and let $f:X\to Y$ be a continuous map. Then there exist canonical isomorphisms
\[\Hom_{\cat{Sh}(X;\cat{k})}(f^{-1}G,F)\cong \Hom_{\cat{Sh}(Y;\cat{k})}(G,f_{*}F)\]
and 
\[\Hom_{\cat{CoSh}(Y;\cat{k})}(f_{\dagger}F,G)\cong \Hom_{\cat{CoSh}(X;\cat{k})}(F,f^{-1}G)\]
for any functor $G:\cat{Y}\to \cat{Vect}_{\cat{k}}$ and any functor $F:\cat{X}\to \cat{Vect}_{\cat{k}}$. In other words, $f^{-1}$ is the left adjoint of $f_*$ and $f_{\dagger}$ is the left adjoint of $f^{-1}$. Consequently, the functor $f^{-1}$ is exact, $f_{\dagger}$ is right exact and $f_{*}$ is left exact.
\end{theorem}

We will denote by $Rf_*$ and $Lf_{\dagger}$ the right derived functor of $f_*$ and the left derived functor of $f_{\dagger}$ respectively.

\begin{definition}
\label{def:external_product_on_prosets}
Let $(P,\le)$ be a preordered set. Let $F,G:\cat{P}\to \mathbf{Vect}_{\mathbf{k}}$ be functors. We define $(F\shtensor G):\cat{P}\to \mathbf{Vect}_{\mathbf{k}}$ to be the functor $(F\shtensor G)_x:=F_x\otimes_{\mathbf{k}}G_x$ and call it the \emph{sheaf tensor product of $F$ and $G$}. As the name suggest one can check that this definition coincides with the usual sheaf tensor product definition when we are considering $F$ and $G$ as sheaves on $P$ with the Alexandrov topology. Furthermore, this definition also coincides with the cosheaf tensor product definition when we consider $F$ and $G$ as cosheaves on $P$ with the opposite Alexandrov topology. Let $P\times P$ be endowed with the product order induced from $(P,\le)$. Define the functor $F\boxtimes G:\cat{P}\times \cat{P}\to \mathbf{Vect}_{\mathbf{k}}$ by $F\boxtimes G=\pi_1^{-1}F\shtensor \pi_2^{-1}G$ where $\pi_i$ are the canonical projections for $i=1,2$. We call $F\boxtimes G$ \emph{the external tensor product of $F$ and $G$}. As the name suggests, this definition coincides with the external tensor product of $F$ and $G$ when considered as sheaves  (\cref{section:background_sheaves}). In general, it is not possible to define an external tensor product for cosheaves as the inverse image functors are not defined as cosheafification is not guaranteed to exist. However for cosheaves of vector thanks to \cref{theorem:cosheafification}  we do not have to worry about that. Furthermore, because $-\shtensor-$ is the same for sheaves and cosheaves and because the inverse image functors for sheaves and cosheaves on pre-ordered sets agree (\cref{example:inverse_image_cosheaf_on_a_poset}) it turns out that we can define $-\boxtimes-$ for cosheaves and that it agrees with the sheaf definition (on preordered sets). Because we are also working with sheaves and cosheaves of vector spaces we have that $-\shtensor-$ is exact. As inverse image functors are also exact it follows that $-\boxtimes-$ is an exact functor. Thus, when we consider it in the derived setting we don't have to write $-\boxtimes^L-$ for its left derived functor, as one usually has to in the general case.
\end{definition}

\subsection{Derived category}
Let $\cat{A}$ be an abelian category. Suppose $\cat{A}$ has enough projectives and injectives. We assume the reader is familiar with the category of chain complexes of objects in $\cat{A}$, $C(\cat{A})$. This section is dedicated to understanding the morphisms in the derived category of $\cat{A}$. This will be necessary in our proof of Theorem 1.1. This exposition to derived categories is borrowed from Steffen Oppermann's notes \cite{oppermann2016homological}.

\begin{definition}
\label{def:quasi_isomorphism}
Let $X$ and $Y$ be in $C(\cat{A})$ and let $f\in \Hom_{c(\cat{A})}(X,Y)$. The chain map $f$ is a \emph{quasi-isomorphism} if for all $j\in \mathbb{Z}$, $H^j(f)$ is an isomorphism.
\end{definition}

\begin{definition}
\label{def:roof}
Let $X$ and $Y$ be in $C(\cat{A})$. A \emph{roof from $X$ to $Y$} is a diagram in $C(\cat{A})$ of the form:
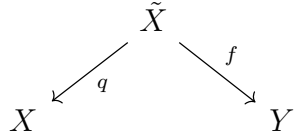
\begin{figure}[H]
\centering
\begin{tikzcd}
&\tilde{X}\arrow[dl,"q"]\arrow[dr,"f"]\\
X&&Y
\end{tikzcd}
\end{figure}
where $q$ is a quasi-isomorphism. We write $f\cdot q^{-1}$ for this roof.
\end{definition}

\begin{lemma}[Ore condition]
\label{lemma:Ore_condition}
Let $X,\tilde{X}$ and $Y$ be in $C(\cat{A})$. Given the solid part in the diagram below, where $q$ is a quasi-isomorphism, it is possible to find the dashed part including $\tilde{Y}$ such that $r$ is also a quasi-isomorphism and that the full diagram commutes in $K(\cat{A})$.
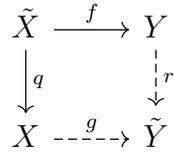
\begin{figure}[H]
\centering
\begin{tikzcd}
\tilde{X}\arrow[r,"f"]\arrow[d,"q"]&Y\arrow[d,dashed,"r"]\\
X\arrow[r,dashed,"g"]&\tilde{Y}
\end{tikzcd}
\end{figure}
Dually, given the dashed part we can find the solid part.
\end{lemma}

\begin{definition}
\label{def:equivalent_roofs}
Let $X$ and $Y$ be in $C(\cat{A})$. Two roofs $f_1\cdot q_1^{-1}$ and $f_2\cdot q_2^{-1}$ are \emph{equivalent} if there is a chain complex $H$ and quasi-isomorphisms $h_i:H\to \tilde{X}_i$ for $i=1,2$, such that the following diagram commutes:
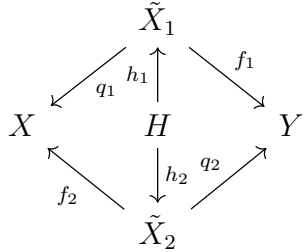
\begin{figure}[H]
\centering
\begin{tikzcd}
&\tilde{X}_1\arrow[dl,"q_1"]\arrow[dr,"f_1"]\\
X&H\arrow[u,"h_1"]\arrow[d,"h_2"]&Y\\
&\tilde{X}_2\arrow[ur,"q_2"]\arrow[ul,"f_2"]
\end{tikzcd}
\end{figure}
It is a standard result that this is indeed an equivalence relation between roofs from $X$ to $Y$.
\end{definition}

We now recall the composition of roofs.

\begin{definition}
\label{def:roof_composition}
Let $X,Y$ and $Z$ be in $C(\cat{A})$. Suppose $f\cdot q^{-1}$ is a roof from $X$ to $Y$ and that $g\cdot r^{-1}$ is a roof from $Y$ to $Z$. The \emph{composition of roofs} $f\cdot q^{-1}$ and $g\cdot r^{-1}$ is defined the be the equivalence class of the roof $g\tilde{f}\cdot (q\tilde{r})^{-1}$ where $\tilde{f}$ and $\tilde{r}$ exist by \cref{lemma:Ore_condition}. More specifically, the composition is a roof from $X$ to $Z$ given by the following diagram:
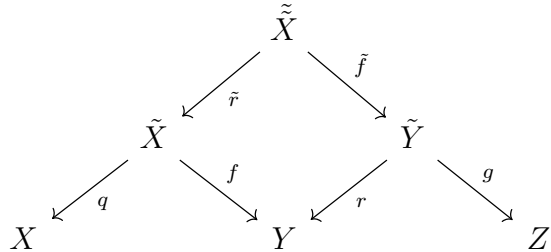
\begin{figure}[H]
\centering
\begin{tikzcd}
&&\tilde{\tilde{X}}\arrow[dl,"\tilde{r}"]\arrow[dr,"\tilde{f}"]\\
&\tilde{X}\arrow[dl,"q"]\arrow[dr,"f"]&&\tilde{Y}\arrow[dl,"r"]\arrow[dr,"g"]\\
X&&Y&&Z
\end{tikzcd}
\end{figure}
\end{definition}

\begin{definition}
\label{def:derived_category}
The \emph{derived category} $D(\cat{A}$ of $\cat{A}$ is given by:
\begin{itemize}
\item $\text{Ob}(D(\cat{A}))=\text{Ob}(C(\cat{A}))$.
\item For $X,Y$ in $\text{Ob}(D(\cat{A}))$, $\Hom_{D(\cat{A})}(X,Y)=\dfrac{\{\text{roofs from X to Y}\}}{\text{equivalence of roofs}}$
\end{itemize}
\end{definition}

To each object $X$ in $\cat{A}$ we can associate to it the corresponding complex concentrated in degree $0$, and to each $f\in \Hom_{\cat{A}}(X,Y)$ we can associate to it the equivalence class  of the roof $f\cdot \mathbf{1}_X^{-1}$. This yields a fully faithful \emph{localization functor} $\iota_{\cat{A}}:\cat{A}\to D(\cat{A})$. More generally, if $f:X\to Y$ is a chain map in $C(\cat{A})$ we can associate to it the equivalence class of the roof $f\cdot \cat{1}_{X}^{-1}$, although this is no longer a fully faithful localization.

\section{Convolution of persistence modules}
\label{section:convolution}

In this section we define sheaf and cosheaf convolutions of derived complexes of persistence modules. We will use these two operations to define convolution distances of derived complexes of persistence modules, that turn out to be equivalent. We show a relationship between this convolution distance and the classical interleaving distance. For the sake of generality, we will be working with a preordered set with a compatible abelian group structure $(P,\le,+,0)$. However all the examples we calculate will be for the case $P=\mathbb{R}$ and the convolution distance will be defined for the case $P=\mathbb{R}^n$. We assume $P$ is given the Alexandrov topology induced by $\le$, or the opposite topology. We denote these topological spaces by $P_{\le}$ and by $P_{\le^{op}}$ respectively. In particular, the open sets in $P_{\le}$ are the up-sets and the open sets in $P_{\le^{op}}$ are the down-sets with respect to $\le$. We denote by $s:P\times  P\to P$ the addition map $s(x,y)=x+y$. An important observation is that $s: P_{\le}\times P_{\le}\to P_{\le}$ and $s:P_{\le^{op}}\times P_{\le^{op}}\to P_{\le^{op}}$ are both continuous maps of topological spaces. This follows from the compatibility of the operation $+$ with the preorder $\le$. Therefore we can talk about the direct image functors $s_{\dagger}$ and $s_*$ and their derived functors $Ls_{\dagger}$ and $Rs_*$. We also denote by $\pi_i:P\times P\to P$ for $i=1,2$ the projections onto the first and second coordinates respectively.

\subsection{Cosheaf convolution}
Here we restrict ourselves to the cosheaf point of view of persistence modules as in \cref{def:persistence_modules_as_cosheaves} and we define the cosheaf convolution of bounded derived complexes of persistence modules.

\begin{definition}
\label{def:convolution_of_cosheaves_poset_topology}
Let $M$ and $N$ be persistence modules, define $M\bullet N:=s_{\dagger}(M\boxtimes N)$. If $M$ and $N$ are in $D^b(\cat{Vect}_{\cat{k}}^{\cat{P}})$,  we define the \emph{cosheaf convolution} to be the left derived functor 
\[M\bullet^L N:= L s_{\dagger}(M\boxtimes N).\]
\end{definition}

\begin{theorem}[\textbf{Theorem 1.2}]
\label{theorem:convolution_of_cosheaves_is_graded_tensor}
Let $M$ and $N$ be in $D^b(\cat{Vect}_{\cat{k}}^{\cat{P}})$. Then there is a canonical isomorphism $M\grtensor^L N\cong M\bullet^L N$.
\end{theorem}

\begin{proof}
Let $x\in \mathbb{R}^n$ and suppose that $M$ and $N$ are persistence modules. Then we have the following canonical isomorphisms
\begin{gather*}
(s_{\dagger}(M\boxtimes N))_x\cong s_*(M\boxtimes N)(D_x):=(M\boxtimes N)(s^{-1}(D_x))\cong \colim\limits_{a+b\le x}(M\boxtimes N)_{(a,b)}\\
:=\colim\limits_{a+b\le x}(\pi_1^{-1}M\shtensor \pi_2^{-1}N)_{(a,b)}\cong\colim\limits_{a+b\le x}(M_{\pi_1(a,b)}\otimes_{\cat{k}} N_{\pi_2(a,b)})=\colim\limits_{a+b\le x}(M_a\otimes_{\cat{k}} N_b) 
\end{gather*}
Recall that $(M\grtensor N)_x:=\colim_{a+b\le x}(M_a\otimes_{\mathbf{k}} N_b)$. Thus, $M\bullet N\cong M\grtensor N$. The derived statement $M\grtensor^L\cong M\bullet^LN$ follows immediately by the uniqueness property of derived functors, see for example \cite{grothendieck1957quelques}.
\end{proof}

Due to \cref{theorem:convolution_of_cosheaves_is_graded_tensor} and known computations for $\grtensor$ between interval modules, see for example \cite{bubenik2019homological}, we have the following proposition.

\begin{proposition}[\textbf{Second part of Proposition 1.4}]
\label{prop:sheaf_convolution_of_interval_modules} Let $P=\mathbb{R}$ and consider two interval modules, $M=\mathbf{k}[a,b)$ and $N=\mathbf{k}[c,d)$. Then $(M\bullet^LN)_0\cong M\bullet N\cong \mathbf{k}[a+c,\min(b+c,a+d))$, $(M\bullet^LN)_1\cong \mathbf{k}[\max(a+d,b+c),b+d)$ and $(M\bullet^LN)_i=0$ for $i\ge 2$.
\end{proposition}

From now on, we will denote $\bullet$ by $\grtensor$ in the remainder of the paper, the exception being \cref{corollary:graded_module_adjunction}. Now let $P=\mathbb{R}^n$ and note that by \cref{prop:classification_into_flats_and_injectives} for $s\in \mathbb{R}^n$ the interval module $\cat{k}[U_s]$ is $\grtensor$-flat. Thus we can write $-\grtensor \cat{k}[U_s]$ instead of $-\grtensor^L\cat{k}[U_s]$ from now on in the $P=\mathbb{R}^n$ case. The following lemma tells us that this operation recovers the translation operation on complexes of persistence modules by $s$.

\begin{lemma}
\label{lemma:thickening_alexandrov_cosheaf}
Let $M$ be in $D(\cat{Vect}_{\cat{k}}^{\cat{R}^n})$. Then there is a canonical isomorphism $M\grtensor \mathbf{k}[U_{s}]\cong M(-s)$. In particular, for every down-set $D$ we have a canonical isomorphism $M(D(-s))\cong (M\grtensor \mathbf{k}[U_{s}])(D)$.
\end{lemma}

\begin{proof}
Suppose that $M$ is a persistence module. We have the following natural isomorphisms: 
$(M\grtensor\mathbf{k}[U_s])_x\cong \colim_{a+b\le x}(M_a\otimes_{\mathbf{k}}\mathbf{k}[U_s]_b)$. Note that if $s\not\le b$, then $\mathbf{k}[U_s]_b=0$. Thus, to evaluate the above colimit, we only need consider the case $s\le b$. Then $a+s\le a+b$. Furthermore, assuming $a+b\le x$, we get $a+s\le x$. Thus, we have $(M\grtensor \mathbf{k}[U_s])_x\cong \colim_{a+s\le x}(M_a\otimes_{\mathbf{k}}\mathbf{k})\cong \colim_{a\le x-s}M_a\cong M_{x-s}$. Thus $M\grtensor\mathbf{k}[U_{s}]\cong M(-s)$.
\end{proof}

\begin{proposition}
\label{prop:properties_of_convolving_object_cosheaf}
Let $\epsilon,\delta\in \mathbb{R}$ and let $M$ be in $D^b(\cat{Vect}_{\cat{k}}^{\cat{R}^n})$.
\begin{itemize}
\item[1.] There are natural isomorphisms $(M\grtensor \mathbf{k}[U_{\bm{\epsilon}}])\grtensor \mathbf{k}[U_{\bm{\delta}}]\cong M\grtensor \mathbf{k}[U_{\bm{\epsilon+\delta}}]$ and $M\grtensor\mathbf{k}[U_0]\cong M$ in $D^b(\cat{Vect}_{\cat{k}}^{\cat{R}^n})$.
\item[2.] If $\delta\ge  \epsilon$, there is a canonical morphism $\mathbf{k}[U_{\bm{\delta}}]\to \mathbf{k}[U_{\bm{\epsilon}}]$ in $D^b(\cat{Vect}_{\cat{k}}^{\cat{R}^n})$ inducing a canonical morphism $M\grtensor \mathbf{k}[U_{\bm{\delta}}]\to M\grtensor \mathbf{k}[U_{\bm{\epsilon}}]$. 
\item[3.] The canonical morphism $M\grtensor \mathbf{k}[U_{\bm{\delta}}]\to M\grtensor \mathbf{k}[U_{\bm{\epsilon}}]$ induces an isomorphism $\mathbb{L}(\mathbb{R}^n;M\grtensor \mathbf{k}[U_{\bm{\delta}}])\to \mathbb{L}(\mathbb{R}^n;M\grtensor \mathbf{k}[U_{\bm{\epsilon}}])$ and hence an isomorphism  in homology.
\end{itemize}
\end{proposition}

\begin{proof}
\begin{itemize}
\item[1.] This rest follows immediately by \cref{lemma:thickening_alexandrov_cosheaf}.
\item[2.] Note that for $\delta\ge \epsilon$, $\mathbf{k}[U_{\bm{\delta}}]$ is a submodule of $\mathbf{k}[U_{\bm{\epsilon}}]$ and the inclusion map is the canonical morphism. By \cref{lemma:thickening_alexandrov_cosheaf}, we have that $M\grtensor\cat{k}[U_{\bm{\delta}}]\cong M(-\bm{\delta})$ and $M\grtensor\cat{k}[U_{\bm{\epsilon}}]\cong M(-\bm{\epsilon})$. Furthermore, the morphism induced is canonically isomorphic to the chain map $M(-\bm{\delta})\to M(-\bm{\epsilon})$ whose $n$-degree component is the natural transformation with components $M^n_{x\le x+\bm{\delta}-\bm{\epsilon}}$ for $x\in \mathbb{R}^n$. The morphism in $D^b(\cat{Vect}_{\cat{k}}^{\cat{R}^n})$ is the localization of this chain map.
\item[3.] By \cref{lemma:thickening_alexandrov_cosheaf}, we have that $M\grtensor\cat{k}[U_{a}]\cong M(-a)$ for all $a\in \mathbb{R}^n$. Thus, if $P\to M$ is a projective resolution of $M$, $P(-\bm{\delta})$ and $P(-\bm{\epsilon})$ are projective resolutions of $M(-\bm{\delta})$ and $M(-\bm{\epsilon})$ respectively since $-\grtensor \mathbf{k}[U_a]$ is an exact functor for all $a\in \mathbb{R}^n$ by \cref{prop:classification_into_flats_and_injectives}. By definition, for any $a\in \mathbb{R}^n$ we have $\mathbb{L}(\mathbb{R}^n;M\grtensor\cat{k}[U_a])\cong \mathbb{L}(\mathbb{R}^n;M(-a)):=L(\mathbb{R}^n;P(-a))$. On the other hand $L(\mathbb{R}^n;P(-a))$ is the complex of $\mathbf{k}$-vector spaces $P(-a)(\mathbb{R}^n):=\colim_{x\in \mathbb{R}^n}P(-a)_x=\colim_{x\in \mathbb{R}^n}P_{x-a}=\colim_{x\in \mathbb{R}^n}P_x=P(\mathbb{R}^n)$. Furthermore, the maps from part 2 do induce isomorphisms as we would be computing the colimits $\colim_{x\in \mathbb{R}^n}P^n_{x\le x+\bm{\delta}-\bm{\epsilon}}$ and these would yield the identity chain morphism $P(\mathbb{R}^n)\to P(\mathbb{R}^n)$.
\end{itemize}
\end{proof}

\subsection{Sheaf convolution}
Here we restrict ourselves to the sheaf point of view of persistence modules as in \cref{def:persistence_modules_as_sheaves} and we define the sheaf convolution of bounded derived complexes of persistence modules.

\begin{definition}
\label{def:convolution_of_sheaves_poset_topology}
Let $M$ and $N$ be persistence modules, define $M* N:=s_{*}(M\boxtimes N)$. If $M$ and $N$ are in $D^b(\cat{Vect}_{\cat{k}}^{\cat{P}})$,  we define the \emph{sheaf convolution} to be the right derived functor 
\[M*^R N:= R s_{*}(M\boxtimes N).\]
\end{definition}

The following proposition gives us a formula for attempting to compute the sheaf convolution. We will see that this formula is useful for computing examples in the case $P=\mathbb{R}$ for persistence modules that decompose as a direct sum of interval modules.

\begin{proposition}
\label{prop:convolution_of_sheaves_poset_topology}
For all $x\in P$, there is a natural isomorphism $(M*N)_x=\lim_{a+b\ge x}(M_a\otimes N_b)$.
\end{proposition}

\begin{proof}
We have the following natural isomorphisms
\begin{gather*}
(M*N)_x:=(s_*(M\boxtimes N))_x\cong s_*(M\boxtimes N)(U_x):=(M\boxtimes N)(s^{-1}(U_x)):=\lim_{a+b\ge x}(M\boxtimes N)_{(a,b)}\\
:=\lim\limits_{a+b\ge x}(\pi_1^{-1}M\shtensor \pi_2^{-1}N)_{(a,b)}\cong \lim\limits_{a+b\ge x}(M_{\pi_1(a,b)}\otimes_{\cat{k}} N_{\pi_2(a,b)})=\lim\limits_{a+b\ge x}(M_a\otimes_{\cat{k}} N_b)
\end{gather*}
\end{proof}

\begin{proposition}
\label{prop:convolution_is_symmetric}
For every pair of persistence modules $M$ and $N$ there is a natural isomorphism $M*N\cong N*M$.
\end{proposition}

\begin{proof}
Let $x\in P$. Then $(M*N)_x\cong \lim\limits_{a+b\ge x}(M_a\otimes N_b)\cong \lim\limits_{a+b\ge x}N_b\otimes_{\cat{k}} M_a\cong (N*M)_x$.
\end{proof}

Now let $P=\mathbb{R}^n$ and note that by \cref{prop:classification_into_flats_and_injectives} for $s\in \mathbb{R}^n$ the interval module $\cat{k}[D_s]$ is injective. Thus we can write $-* \cat{k}[U_s]$ instead of $-*^R\cat{k}[U_s]$ from now on in the $P=\mathbb{R}$ case. The following lemma tells us that this operation recovers the translation operation on complexes of persistence modules by $s$.

\begin{lemma}
\label{lemma:thickening_alexandrov_sheaf}
Let $M$ be in $D(\cat{Vect}_{\cat{k}}^{\cat{R}^n})$. Then there is a canonical isomorphism $M* \cat{k}[D_{s}]=M(-s)$. In particular, for every up-set $U$ we have a canonical isomorphism $M(U(-s))\cong (M* \cat{k}[D_{s}])(U)$. 
\end{lemma}

\begin{proof}
The proof is dual to the proof of \cref{lemma:thickening_alexandrov_cosheaf}.
\end{proof}

%


\begin{proposition}
\label{prop:properties_of_convolving_object_sheaf}
Let $\epsilon,\delta\in \mathbb{R}$ and let $M$ be in $D^b(\cat{Vect}_{\cat{k}}^{\cat{R}^n})$.
\begin{itemize}
\item[1.] There are natural isomorphisms $(M* \mathbf{k}[D_{\bm{\epsilon}}])* \mathbf{k}[D_{\bm{\delta}}]\cong M* \mathbf{k}[D_{\bm{\epsilon+\delta}}]$ and $M* \mathbf{k}[D_0]\cong M$ in $D^b(\cat{Vect}_{\cat{k}}^{\cat{R}^n})$.
\item[2.] If $\delta\ge  \epsilon$, there is a canonical morphism $\mathbf{k}[D_{\bm{\delta}}]\to \mathbf{k}[D_{\bm{\epsilon}}]$ in $D^b(\cat{Vect}_{\cat{k}}^{\cat{R}^n})$ inducing a canonical morphism $M* \mathbf{k}[D_{\bm{\delta}}]\to M* \mathbf{k}[D_{\bm{\epsilon}}]$. 
\item[3.] The canonical morphism $M* \mathbf{k}[D_{\bm{\delta}}]\to M* \mathbf{k}[D_{\bm{\epsilon}}]$ induces an isomorphism $R\Gamma(\mathbb{R}^n,M* \mathbf{k}[D_{\bm{\delta}})\to R\Gamma(\mathbb{R}^n,M* \mathbf{k}[D_{\bm{\epsilon}}])$ and hence an isomorphism  in cohomology. 
\end{itemize}
\end{proposition}

\begin{proof}
Use \cref{lemma:thickening_alexandrov_sheaf} and dual arguments to the ones in the proof of \cref{prop:properties_of_convolving_object_cosheaf}.
\end{proof}

\begin{proposition}[\textbf{First part of Proposition 1.4}]
\label{prop:sheaf_convolution_of_interval_modules} Let $P=\mathbb{R}$ and consider two interval modules, $M=\mathbf{k}[a,b)$ and $N=\mathbf{k}[c,d)$. Then $(M*^RN)^0\cong M*N\cong \mathbf{k}[\max(a+d,b+c),b+d)$, $(M*^RN)^1\cong \mathbf{k}[a+c,\min(b+c,a+d))$ and $(M*^RN)^i=0$ for $i\ge 2$.
\end{proposition}

\begin{proof}
It follows by \cref{prop:convolution_of_sheaves_poset_topology} that $(M*N)_x$ is the $\mathbf{k}$-vector subspace of $\prod_{s+t=r}M_a\otimes_{\mathbf{k}}N_b$ that is the limit of the digram of vector spaces $M_p\otimes_{\mathbf{k}}N_q$ with $p+q\ge x$ and the maps in the diagram are given by $\mathbf{1}_{M_{p}}\otimes_{\mathbf{k}}N_{q\le q'}$ and $M_{p\le p'}\otimes_{\mathbf{k}}\mathbf{1}_{N_q}$ and their compositions. Suppose that $b+c\le a+d$ as in \cref{fig:2}. Note that $M_a\otimes_{\mathbf{k}}N_c\cong \mathbf{k}$. However, for any other $a'\neq a$, $c'\neq c$ such that $a'+c'=a+c$ we have $M_{a'}\otimes_{\mathbf{k}}N_{b'}=0$. Therefore,  we have $(M*N)_{a+c}=\lim_{s+t\ge a+c}M_s\otimes_{\mathbf{k}}N_t=0$. Same reasoning shows that along any antidiagonal $l<a+d$ we have $(M*N)_l=\lim_{s+t\ge l}M_s\otimes_{\mathbf{k}}N_t=0$. The antidiagonal $a+d$ is the first one where we have a nontrivial limit as all the relevant maps that contribute to the limit computation outside the rectangle in \cref{fig:2} are mapping into trivial vector spaces. Furthermore $\lim_{s+t\ge a+d}(M_s\otimes_{\mathbf{k}}N_t)\cong (M*N)_{a+d}\cong \mathbf{k}$. This is because all the vector spaces in the rectangle are copies of $\mathbf{k}$ and all the linear maps are identities and thus it is the one dimensional diagonal subspace of $\prod_{s+t=a+d}M_s\otimes_{\mathbf{k}}N_t$ that is the limit. Similar reasoning shows that for any antidiagonal $l<b+d$ we have $(M*N)_l\cong \mathbf{k}$. Finally, at $b+d$ we have $(M*N)_{b+d}=0$ as all the vector spaces relevant in the limit computation are trivial. Thus $M*N\cong \mathbf{k}[a+d,b+d)$. If we had instead assumed in the beginning that $a+d\le b+c$ by the same reasoning we would have gotten $M*N\cong \mathbf{k}[b+c,b+d)$. Thus, in general we have $M*N\cong \mathbf{k}[\max(a+d,b+c),b+d)$.

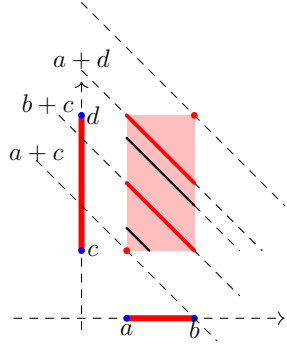
\begin{figure}[h]
\centering
\begin{tikzpicture}[line cap=round,line join=round,,x=1.0cm,y=1.0cm,scale=0.3]
\draw[->,color=black,dashed] (-3,0) -- (9,0);
\draw[->,color=black,dashed] (0,-0.5) -- (0,10.5);
\draw[color=red,line width=0.8mm] (2,0) -- (5,0);
\draw[color=red,line width=0.8mm] (0,3) -- (0,9);
\fill [color=blue] (2,0) circle (1.5mm);
\fill [color=blue] (5,0) circle (1.5mm);
\fill [color=blue] (0,3) circle (1.5mm);
\fill [color=blue] (0,9) circle (1.5mm);

\fill[fill=red!25] (5,3) rectangle (2,9);
\draw[color=black,dashed] (-2,7) -- (6,-1);
\draw[color=black,dashed] (-1,9) -- (7,1);
\draw[color=black,dashed] (0,11)--(8,3);
\draw[color=black,dashed] (0,14)--(9,5);
\draw[color=red, line width=0.5mm] (2,6)--(5,3);
\draw[color=red, line width=0.5mm] (2,9) -- (5,6);
\fill [color=red] (2,3) circle (1.5mm);
\fill [color=red] (5,9) circle (1.5mm);
\draw[color=black,line width=0.3mm] (2,4)--(3,3);
\draw[color=black,line width=0.3mm](2,8)--(5,5);
\draw[color=black,dashed] (5,5)--(7,3);
\draw[color=black] (2,-0.5) node[scale=0.8] {$a$};
\draw[color=black] (5,-0.5) node[scale=0.8] {$b$};
\draw[color=black] (0.5,3) node[scale=0.8] {$c$};
\draw[color=black] (0.5,9) node[scale=0.8] {$d$};
\draw[color=black] (-2,7.3) node[scale=0.8]{$a+c$};
\draw[color=black] (-1.5,9.5) node[scale=0.8] {$b+c$};
\draw[color=black] (0,11.5) node[scale=0.8] {$a+d$};
\end{tikzpicture}
\caption {\rmfamily Sheaf convolution of interval modules : $\mathbf{k}[a,b)* \mathbf{k}[c,d) =\mathbf{k}[\max(a+d,b+c)$,$b+d)$ }
\label{fig:2}
\end{figure} 

In order to calculate $(M*^RN)^1$ we note that by \cref{prop:classification_into_flats_and_injectives} we have the following augmented injective resolution of $M=\mathbf{k}[a,b)$:
\[0\to \mathbf{k}[a,b)\to \mathbf{k}(-\infty,b)\to \mathbf{k}(-\infty,a)\to 0
\]
Apply the functor $-* \mathbf{k}[c,d)$ to the injective resolution to get the following (no longer exact) sequence.
\[0\to \mathbf{k}[b+c,b+d)\to \mathbf{k}[a+c,a+d)\to  0
\]
Calculating cohomology in degree 1 we get the following:
\[((\mathbf{k}[a,b)*^R\mathbf{k}[c,d)))^1=\mathbf{k}[a+c,\min(b+c,a+d))\,.\]
Similarly
$(\mathbf{k}[a,b)*^R\mathbf{k}(-\infty,d))^1 =0$
and 
$(\mathbf{k}[a,\infty)*^R\mathbf{k}[c,d))^1 = \mathbf{k}[a+c,a+d)$. 
\end{proof}

Recall the definition of the sheaf hom $\scHom$. If $F$ and $G$ are sheaves on a space $X$, $\scHom(F,G)$ is the sheaf of abelian groups $\scHom(F,G)(U)=\Hom(F|_U,G|_U)$ for every open $U\subseteq X$, where $F|_U$ and $G|_U$ are the restrictions of $F$ and $G$ on $U$ respectively.

\begin{definition}
\label{def:graded_module_hom_as_a_sheaf}
Let $M,N:\cat{P}\to \cat{Vect}_{\cat{k}}$ be persistence modules. We can define $\scHom^*(M,N):=\pi_{2*}\scHom(\pi_1^{-1}M,s^{-1}N)$. If $M$ and $N$ are in $D^b(\cat{Vect}_{\cat{k}}^{\cat{P}})$, we also have the derived functor $R\scHom^*(M,N):=R\pi_{2*}\scHom(\pi_1^{-1}M,s^{-1}N)$.
\end{definition}

\begin{theorem}[\textbf{Theorem 1.3}]
\label{theorem:graded_module_hom_as_a_sheaf}
Let $M$ and $N$ be in $D^b(\cat{Vect}_{\cat{k}}^{\cat{P}})$. There is a canonical isomorphism $R\scHom^*(M,N)\cong R\uHom (M,N)$. 
\end{theorem}

\begin{proof}
Let $x\in P$ and suppose first that $M$ and $N$ are persistence modules. We have the following natural isomorphisms
\begin{gather*}
\scHom^*(M,N)_x\cong \scHom^*(M,N)(U_x):=\pi_{2*}\scHom(\pi_1^{-1}M,s^{-1}N)(U_x):=\\
:=\scHom(\pi_1^{-1}M,s^{-1}N)(\pi_2^{-1}(U_x)):=\Hom(\pi_1^{-1}M|_{\pi_2^{-1}(U_x)},s^{-1}N|_{\pi_2^{-1}(U_x)})\cong\\
\cong\lim\limits_{b\ge x,d\ge x,c-a\ge 0,d-b\ge 0 }\Homk((\pi_1^{-1}M|_{\pi_2^{-1}(U_x)})_{(a,b)},(s^{-1}N|_{\pi_2^{-1}(U)})_{(c,d)})\cong\\
\cong \lim\limits_{d\ge b\ge x,c-a\ge 0}\Homk (M_a,N_{c+d})\cong \lim\limits_{d+c-a\ge x} \Homk(M_a,N_{c+d})=\uHom(M,N)_x
\end{gather*}
The crucial step in the above composition of canonical isomorphisms was the observation 
\begin{gather*}
\Hom(\pi_1^{-1}M|_{\pi_2^{-1}(U_x)},s^{-1}N|_{\pi_2^{-1}(U_x)})\cong \\
\cong\lim\limits_{b\ge x,d\ge x,c-a\ge 0,d-b\ge 0 }\Homk((\pi_1^{-1}M|_{\pi_2^{-1}(U_x)})_{(a,b)},(s^{-1}N|_{\pi_2^{-1}(U)})_{(c,d)}),
\end{gather*}
which follows from the remark in \cref{def:internal_hom}. The derived statement $R\scHom^{*}(M,N)\cong R\uHom(M,N)$ follows immediately by the uniqueness property of derived functors, see for example \cite{grothendieck1957quelques}.
\end{proof}

The bifunctor $-\grtensor -$ is the left adjoint of $\uHom(-,-)$, see \cite{bubenik2019homological}. Note that the formulation of $R\scHom^{*}$ is very much analogous in symbols to the definition of the right adjoint of the classical convolution functor, $\scHom^{\star}$ from \cref{def:convolution_adjoint} introduced by Tamarkin in \cite{tamarkin2013microlocal}, however the underlying formula is very different. For instance, we do not use the exceptional inverse image functor of the addition map in our definition $s^{!}$. Recall that we showed that our cosheaf convolution is canonically isomorphic to the graded module tensor product, an operation well studied in graded module theory. Now we have shown its adjoint, $\uHom$, another classical operation in graded module theory, used for example to define the Matlis dual of a module, has a canonical sheaf theoretic interpretation, that has some resemblance to functors studied in microlocal analysis. This adjointness is extended in the derived setting as persistence modules are a Grothendieck category and thus have enough injectives and projectives. We thus have the following corollary. 

\begin{corollary}
\label{corollary:graded_module_adjunction}
For every triple of persistence modules $M,N$ and $P$ there exist canonical isomorphisms
\[\Hom_{D^{b}(\cat{Vect}_{\cat{k}}^{\cat{P}})}(M\bullet^L N,P)\cong \Hom_{D^{b}(\cat{Vect}_{\cat{k}}^{\cat{P}})}(M,R\scHom^*(N,P)).\]
\end{corollary}

Observe that \cref{def:convolution_of_cosheaves_poset_topology,def:convolution_of_sheaves_poset_topology} are analogous to \cref{def:convolution_of_constructible_sheaves} and that \cref{prop:properties_of_convolving_object_sheaf,prop:properties_of_convolving_object_cosheaf} are analogous to \cref{prop:properties_of_convolving_object_constr_sheaf}. We believe this justifies  calling the operations we defined as convolutions. However, the reader should note that in our convolution definitions we are considering the direct image of the addition map $s$, while the convolution of sheaves involves the direct image with proper supports of the addition map $s$ \cref{section:background_sheaves}. Our choice of only considering the direct image and not direct image with proper supports was because of the following two reasons. A useful tool for computing the direct image with proper supports comes in the form of Beck-Chevaley Theorem (\cref{theorem:beck_chevaley}). The theorem states that computing the direct image with proper supports is equivalent to computing the compactly supported cohomology of the fibers. However, the topological spaces involved are assumed to be Hausdorff (and locally compact). Note that $\mathbb{R}^n$ is not a Hausdorff space with respect to the Alexandrov topology (for any $a,b\in \mathbb{R}^n$, $U_a\cap U_b\neq\emptyset$ and $D_a\cap D_b\neq \emptyset$). Hence, even if we had considered $s_!$ in our definition, we would have no easy way of computing the convolution for most of the examples of interest to us. The second reason is \cref{theorem:convolution_of_cosheaves_is_graded_tensor} which says that our cosheaf convolution is canonically isomorphic to the derived graded module tensor product, a canonical operation in graded module theory. This enforces our belief that we are on the right track with the cosheaf definition.

\subsection{Convolution distance}
In this section we assume $P=\mathbb{R}^n$ and use the sheaf and cosheaf convolutions of derived complexes of persistence modules to define a convolution distance. 

\begin{definition}
\label{def:convolution_distance_sheaf}
Let $M$ and $N$ be in $D^b(\cat{Vect}_{\cat{k}}^{\cat{R}^n})$ and let $\epsilon\ge 0$. We say $M$ and $N$ are $\epsilon$-isomorphic if there are morphisms $f:\mathbf{k}[D_{\bm{\epsilon}}]* M\to N$ and $g:\mathbf{k}[D_{\bm{\epsilon}}]* N\to M$ in $D^b(\cat{Vect}_{\cat{k}}^{\cat{R}^n})$ such that the composition $\mathbf{k}[D_{2\bm{\epsilon}}]* M\xrightarrow{\mathbf{k}[D_{\bm{\epsilon}}]* f}\mathbf{k}[D_{\bm{\epsilon}}]* N\xrightarrow{g} M$ coincides with the natural morphism $\mathbf{k}[D_{2\bm{\epsilon}}]* M\to M$ and the composition $\mathbf{k}[D_{2\bm{\epsilon}}]* N\xrightarrow{\mathbf{k}[D_{\bm{\epsilon}}]* g}\mathbf{k}[D_{\bm{\epsilon}}]* M\xrightarrow{g}  N$ coincides with the natural morphism $\mathbf{k}[D_{2\bm{\epsilon}}]*N \to N$. If $M$ and $N$ are $\epsilon$-isomorphic, then they are $\delta$-isomorphic for any $\delta\ge \epsilon$. Thus, one can define
\[d^C(M,N)=\inf(\{+\infty\}\cup \{\epsilon\in \mathbb{R}_{\ge 0}\,|\, \text{M and N are } \epsilon\text{-isomorphic}\})\]
\end{definition}

\begin{definition}
\label{def:convolution_distance_cosheaf}
Let $M$ and $N$ be in $D^b(\cat{Vect}_{\cat{k}}^{\cat{R}^n})$ and let $\epsilon\ge 0$. We say $M$ and $N$ are $\epsilon$-isomorphic if there are morphisms $f:\mathbf{k}[U_{\bm{\epsilon}}]\grtensor M\to N$ and $g:\mathbf{k}[U_{\bm{\epsilon}}]\grtensor N\to M$ in $D^b(\cat{Vect}_{\cat{k}}^{\cat{R}^n})$ such that the composition $\mathbf{k}[U_{2\bm{\epsilon}}]\grtensor M\xrightarrow{\mathbf{k}[U_{\bm{\epsilon}}]\grtensor f}\mathbf{k}[U_{\bm{\epsilon}}]\grtensor N\xrightarrow{g}  M$ coincides with the natural morphism $\mathbf{k}[U_{2\bm{\epsilon}}]\grtensor M\to M$ and the composition $\mathbf{k}[U_{2\bm{\epsilon}}]\grtensor N\xrightarrow{\mathbf{k}[U_{\bm{\epsilon}}]\grtensor g}\mathbf{k}[U_{\bm{\epsilon}}]\grtensor M\xrightarrow{g} N$ coincides with the natural morphism $\mathbf{k}[U_{2\bm{\epsilon}}]\grtensor N \to N$. If $M$ and $N$ are $\epsilon$-isomorphic, then they are $\delta$-isomorphic for any $\delta\ge \epsilon$. Thus, one can define
\[d_C(M,N)=\inf(\{+\infty\}\cup \{\epsilon\in \mathbb{R}_{\ge 0}\,|\, \text{M and N are } \epsilon\text{-isomorphic}\})\]
\end{definition}

\begin{lemma}
\label{lemma:convolution_distances_agree}
Let $M$ and $N$ be in $D^b(\cat{Vect}_{\cat{k}}^{\cat{R}^n})$. Then $d_C(M,N)=d^C(M,N)$.
\end{lemma}

\begin{proof}
By 
\cref{lemma:thickening_alexandrov_cosheaf,lemma:thickening_alexandrov_sheaf} 
we have that for every $a\in \mathbb{R}^n$, $\mathbf{k}\grtensor M\cong M(-a)\cong \mathbf{k}[D_a]*M$. Thus the result follows from the definitions of $d_C$ and $d^C$.
\end{proof}

From now one we use $d_C$ to denote both $d_C$ and $d^C$. Furthemore, by construction $d_C$ is an extended pseudo-metric on $D^b(\cat{Vect}_{\cat{k}}^{\cat{R}^n})$. The authors in \cite{kashiwara2018persistent} use the idea behind the definition of the interleaving distance of persistence modules to define a convolution distance between complexes of sheaves on euclidean space. We now show that in fact the classical interleaving distance is a type of convolution in the sense that the convolution distance extends it in the derived setting.

\begin{theorem}[\textbf{Theorem 1.1}]
\label{theorem:interleaving_is_convolution}
Let $M$ and $N$ be persistence modules which we think of as complexes of persistence modules concentrated in degree $0$. Then $d_C(M,N)=d_I(M,N)$.
\end{theorem}

\begin{proof}
This follows from the definitions and \cref{lemma:thickening_alexandrov_cosheaf} or \cref{lemma:thickening_alexandrov_sheaf} and the fact that the localization functor $\iota_{\cat{Vect}_{\cat{k}}^{\cat{R}^n}}:\cat{Vect}_{\cat{k}}^{\cat{R}^n}\to D(\cat{Vect}_{\cat{k}}^{\cat{R}^n})$ is fully faithful.
\end{proof}

\begin{proposition}
\label{theorem:interleaving_of_complexes}
Let $M$ and $N$ be in $D^b(\cat{Vect}_{\cat{k}}^{\cat{R}^n})$. Suppose that the boundary maps of $M$ and $N$ are $0$. Then
\[d_C(M,N)\le \max_nd_I(M^n,N^n).\]
Without the $0$ boundary maps assumption, we still have
 \[d_C(H^{\bullet}(M),H^{\bullet}(N)) \le\max_nd_I(H^n(M),H^n(N))\].
\end{proposition}

\begin{proof}
If $\max_nd_I(M^n,N^n)=\infty$ we are done. So suppose $\max_nd_I(M^n,N^n)<\infty$, say $\max_nd_I(M^n,N^n)=\epsilon$.  Let $\delta>\epsilon$. Then, the persistence modules $M^n$ and $N^n$ are $\delta$-interleaved for every $n\in \mathbb{Z}$. Thus there exist persistence module morphisms $f^n:M^n(-\bm{\delta})\to N^n$ and $g^n:N^n(-\bm{\delta})\to M^n$ that achieve the $\delta$-interleaving for all $n\in \mathbb{Z}$. As all the boundary maps in the complexes $M$ and $N$ are $0$, the sequences of persistence module morphisms $f^n$ and $g^n$ assemble into chain maps $f:M(-\bm{\delta})\to N$ and $g:N(-\bm{\delta})\to N$. The localizations of the maps $f$ and $g$ thus give the desired $\delta$-isomorphism in $D^b(\cat{Vect}_{\cat{k}}^{\cat{R}^n})$. Thus the result follows. 
\end{proof}

\section{Stability}
\label{section:stability}
In this section we assume $P=\mathbb{R}^n$ and discuss stability results involving the convolution distance of bounded derived complexes of persistence modules using the sheaf and cosheaf points of view. Let $f:X\to \mathbb{R}^n$ be a set-theoretic function. Note that for every $a\le b\in \mathbb{R}^n$ we have inclusion maps $f^{-1}(D_a)\to f^{-1}(D_b)$ and $f^{-1}(U_b)\to f^{-1}(U_a)$. Applying singular homology and cohomology in degree $n$ functors with coefficients in a field $\mathbf{k}$ for every $a\in \mathbb{R}^n$ we obtain vector spaces $H_n(f^{-1}(D_a);\mathbf{k})$ and $H^n(f^{-1}(U_a);\mathbf{k})$. Furthermore, for $a\le b$ the mentioned inclusions maps induce $\mathbf{k}$-linear maps $H_n(f^{-1}(D_a);\mathbf{k})\to H_n(f^{-1}(D_b);\mathbf{k})$ and $H^n(f^{-1}(U_a);\mathbf{k})\to H^n(f^{-1}(U_b);\mathbf{k})$, respectively. Thus, by considering homology and cohomology of inverse images of principal down-sets and up-sets in $\mathbb{R}^n$, respectively, we obtain two persistence modules which we label by $H_nf$ and $H^nf$ respectively. We have the following classical result in persistence theory.

\begin{theorem}{\cite[Theorems 4.3,4.4 and Example 4.6]{bubenik2015metrics}}
\label{theorem:stability}
Let $X$ be a topological space and let $f,g:X\to \mathbb{R}^n$ be set-theoretic functions. Then $d_I(H_{n}f,H_{n}g)\le ||f-g||_{\infty}$ and $d_I(H^{n}f,H^{n}g)\le ||f-g||_{\infty}$.
\end{theorem}

\begin{remark}
\label{remark:stability}
Theorems 4.3, 4.4 and Example 4.6 in \cite{bubenik2015metrics} only give us $d_I(H_nf,H_ng)\le ||f-g||_{\infty}$. However, it is straighforward to adapt the arguments in a dual way to prove the cohomology case as well.
\end{remark}

Given a set theoretic function $f:X\to \mathbb{R}^n$ we can also use $f$ to construct complexes of persistence modules in the following way. For every $a$ we can assing to $f^{-1}(D_a)$ the chain complex of $\mathbf{k}$ vector spaces that is the chain complex of singular chains with coefficients in $\mathbf{k}$, $C_{\bullet}(f^{-1}(D_a);\mathbf{k})$. Similarly we can assign to $f^{-1}(U_a)$ and the cochain complex of $\mathbf{k}$ vector spaces that is the cochain complex of singular cochains with coefficients in $\mathbf{k}$, $C^{\bullet}(f^{-1}(U_a);\mathbf{k})$. Denote these complexes of persistence modules by $C_{\bullet}f$ and $C^{\bullet}f$ respectively. Denote by $H_{\bullet}f$ and by $H^{\bullet}f$ the respective homology and cohomology complexes. Note that $(H_{\bullet}f)_n=H_nf$ and $(H^{\bullet}f)^n=H^nf$. From \cref{theorem:stability} and \cref{theorem:interleaving_of_complexes} we have \cref{corollary:stability_complexes}.

\begin{corollary}
\label{corollary:stability_complexes}
Let $X$ be a topological space and let $f,g:X\to \mathbb{R}^n$ be set-theoretic functions. Suppose $H_{\bullet}f$ and $H^{\bullet}f$ are bounded complexes of persistence modules. Then $d_C(H_{\bullet}f,H_{\bullet}g)\le ||f-g||_{\infty}$ and  $d_C(H^{\bullet}f,H^{\bullet}g)\le ||f-g||_{\infty}$.
\end{corollary}

\subsection{Stability for direct images}
Here we consider stability of direct images of sheaves and cosheaves on $\mathbb{R}^n$.

\begin{theorem}[Stability for direct images]
\label{theorem:stability_for_direct_images}
Let $X$ be a topological space and let $f,g:X\to \mathbb{R}^n_{\le}$ be continuous maps. Let $F$ be in $D^b(\cat{Sh}(X;\cat{k}))$. Then 
\[d_C(Rf_*F,Rg_*F)\le ||f-g||_{\infty}.\]
Dually, if $G$ is in $D^b(\cat{CoSh}(X;\cat{k}))$. Then 
\[d_C(Lf_{\dagger}G,Lg_{\dagger}G)\le ||f-g||_{\infty}.\]
\end{theorem}

\begin{proof}
We prove the sheaf statement. The cosheaf one is dual. If $||f-g||_{\infty}=\infty$ we are done so suppose that $||f-g||_{\infty}=\epsilon<\infty$. Let $F\to E$ be an injective resolution of $F$. We need to show that $d_C(f_*E,g_*E)\le \epsilon$. Let $U\subseteq\mathbb{R}^n$ be an up-set. Since $||f-g||_{\infty}\le \epsilon$ we have that $f^{-1}(U)\subseteq g^{-1}(U(-\bm{\epsilon}))\subseteq f^{-1}(U(-2\bm{\epsilon}))$ and $g^{-1}(U)\subseteq f^{-1}(U(-\bm{\epsilon}))\subseteq g^{-1}(U(-2\bm{\epsilon}))$. By definition, we have $f_*E^n(U):=E^n(f^{-1}(U))$ for all $n\in \mathbb{Z}$. Thus, we have the following sheaf restriction maps that commute with the boundary maps, as $E$ is assumed to be a complex:
\begin{figure}[H]
\centering
\begin{tikzcd}
E^n(f^{-1}(U(-\bm{\epsilon})))\arrow[r]\arrow[d,"\partial"]&E^n(g^{-1}(U(-\bm{\epsilon})))\arrow[r]\arrow[d,"\partial"]& E^n(f^{-1}(U))\arrow[d,"\partial"]\\
E^{n+1}(f^{-1}(U(-\bm{\epsilon})))\arrow[r] & E^{n+1}(g^{-1}(U(-\bm{\epsilon})))\arrow[r]&E^{n+1}(f^{-1}(U))
\end{tikzcd}
\end{figure}
and
\begin{figure}[H]
\centering
\begin{tikzcd}
E^n(g^{-1}(U(-\bm{\epsilon})))\arrow[r]\arrow[d,"\partial"]&E^n(f^{-1}(U(-\bm{\epsilon})))\arrow[r]\arrow[d,"\partial"]& E^n(g^{-1}(U))\arrow[d,"\partial"]\\
E^{n+1}(g^{-1}(U(-\bm{\epsilon})))\arrow[r] & E^{n+1}(f^{-1}(U(-\bm{\epsilon})))\arrow[r]&E^{n+1}(g^{-1}(U))
\end{tikzcd}
\end{figure}
We can thus define a chain maps $\Phi: f_*E(-\bm{\epsilon})\to g_*E$ and $\Psi:g_*E(-\bm{\epsilon})\to f_*E$ where for each up-set $U\subseteq \mathbb{R}^n$, $\Phi^n (U)$ and $\Psi^n(U)$ are the restrictions $E^n(f^{-1}(U(-\bm{\epsilon})))\to E^n(g^{-1}(U))$ and $E^n(g^{-1}(U(-\bm{\epsilon})))\to E^n (f^{-1}(U))$ respectively. Then by construction, the localizations of the chain maps $\Phi$ and $\Psi$ give us an $\epsilon$-isomorphism between $f_*E$ and $g_*E$ and thus $d_C(f_*E,g_*E)\le \epsilon$.
\end{proof}

The natural question to ask is in how many examples of interest in topological data analysis do we actually encounter continuous maps $f:X\to \mathbb{R}^n_{\le}$ or $f:X\to \mathbb{R}^n_{\le^{op}}$. The following example shows that even a somewhat canonical example, like a Morse function on a circle, does not fit in this framework. 

\begin{example}
Let $p:S^1\to \mathbb{R}$ be the projection onto the $x$-axis of the unit circle, centered somewhere on the $y$-axis (\cref{fig:3}). Note that $p:S^1\to \mathbb{R}^n_{\le}$ and $p:S^1\to \mathbb{R}^n_{\le^{op}}$ are not continuous maps. Indeed, $[0,\infty)\subseteq \mathbb{R}$ is open in $\mathbb{R}^n_{\le}$ and the inverse image $p^{-1}([0,\infty))$ is not. Similarly for $(-\infty,0]$ in $\mathbb{R}^n_{\le^{op}}$.
\label{example:sheaf_direct_image_circle}
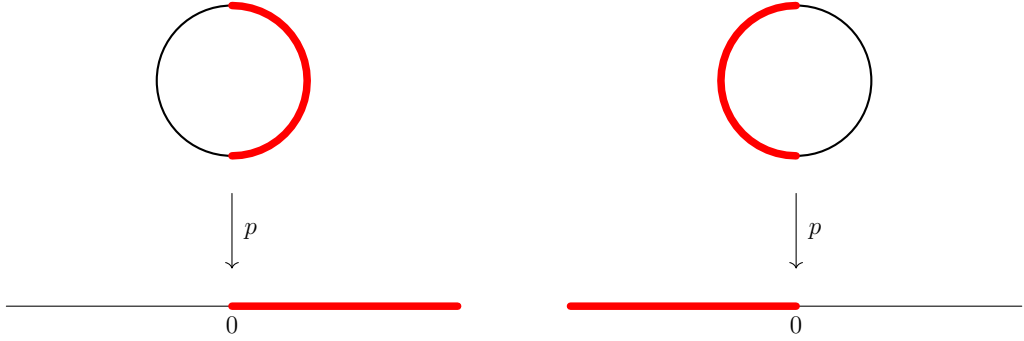
\begin{figure}[h]
\centering
\begin{tikzpicture}[line cap=round,line join=round,,x=1.0cm,y=1.0cm,scale=0.5]
\draw[color=black] (-3,0) -- (9,0);
\draw[black, thick] (3,6) circle (2 cm);
\draw[->,color=black] (3,3)--(3,1);
\draw[color=black] (3.5,2) node[scale=0.8] {$p$};
\draw[color=red,line width=1mm] (3,0) -- (9,0);
\draw[color=black] (3,-0.5) node[scale=0.8] {$0$};
\draw [red, xshift=3cm, yshift=6cm, domain=-90:90,scale=2,line width=1mm] plot(\x:1);
\draw[color=black] (12,0) -- (24,0);
\draw[black, thick] (18,6) circle (2 cm);
\draw[->,color=black] (18,3)--(18,1);
\draw[color=black] (18.5,2) node[scale=0.8] {$p$};
\draw[color=red,line width=1mm] (12,0) -- (18,0);
\draw[color=black] (18,-0.5) node[scale=0.8] {$0$};
\draw [red, xshift=18cm, yshift=6cm, domain=90:270,scale=2,line width=1mm] plot(\x:1);
\end{tikzpicture}
\caption {\rmfamily The complex of persistence module $H^{\bullet}p$ is not equal to $Rp_*F$, where $F$ is a complex of sheaves valued of $\mathbf{k}$-vector spaces on $S^1$ that assigns to each open $U\subseteq S^1$  the singular cohomology complex valued in $\mathbf{k}$, $H^{\bullet}(U;\mathbf{k})$. Furthermore, $Rp_*F$ is not even defined as $p$ is not continuous. Same for $H_{\bullet}p$ and $Lp_{\dagger}F$.}
\label{fig:3}
\end{figure}
\end{example}

Thus, if $p:S^1\to \mathbb{R}^n$ is as in \cref{example:sheaf_direct_image_circle} direct images of sheaves or cosheaves of vector spaces on $S^1$ are not defined. In particular, persistent homology/cohomology persistence modules are not obtained as direct images of homology and cohomology sheaves on $S^1$. This observation somewhat limits the applicability of \cref{theorem:stability_for_direct_images}. We attempt to address this issue in two ways, by modifying the topology on the domain and codomain.

\subsection{Stability for modified direct images on the domain}
Let $f:X\to \mathbb{R}^n_{\le}$ be a set-theoretic map. Denote by $X_f$ the topological space with the underlying set $X$ and the topology the pullback topology induced by $f$. That is, $A\subseteq X_f$ is open if and only if there exists an open $U\subseteq \mathbb{R}^n$ such that $f^{-1}(U)=A$. Similarly we denote by $X_f^{op}$ the pullback topology of a map $f:X\to \mathbb{R}^n_{\le^{op}}$. If $X$ is already a topological space to being with, we will abuse notation and denote by $X_f$ the topological space whose open sets are all the open sets in $X$ and also all the open sets in $X_f$. Same for $X_f^{op}$. Thus, the identity maps $\mathbf{1}_{X}:X_f\to X$ and $\mathbf{1}_X:X_f^{op}\to X$ are continuous as the topologies on $X_f$ and $X_f^{op}$ are finer than the one on $X$ by construction. 

\begin{figure}[H]
\centering
\begin{tikzcd}
X_f\arrow[r,"\mathbf{1}_X"]\arrow[dr,"f"]&X\arrow[d,"f"]&&&X_f^{op}\arrow[r,"\mathbf{1}_X"]\arrow[dr,"g"]&X\arrow[d,"g"]\\
&\mathbb{R}^n_{\le}&&&& \mathbb{R}^n_{\le^{op}}
\end{tikzcd}
\caption{If $F$ is a sheaf on $X$ and $f:X\to \mathbb{R}^n_{\le}$ is not necessarily continuous, we can construct the sheaf $f_*\mathbf{1}_{X}^{-1}F$  on $\mathbb{R}^n_{\le}$. Similarly, if $G$ is a cosheaf on $X$ and $g:X\to \mathbb{R}^n_{\le^{op}}$ is not necessarily continuous, we can construct the cosheaf $g_{\dagger}\mathbf{1}_{\mathbb{R}^n}^{-1}G$ on $\mathbb{R}^n_{\le^{op}}$.}
\label{fig:4}
\end{figure}

If $X$ is a topological space and $f,g:X\to \mathbb{R}^n_{\le}$ or $f,g:X\to\mathbb{R}^n_{\le^{op}}$ are not necessarily continuous maps, we denote by $X_{fg}$ and by $X_{fg}^{op}$ the topologies on $X$ that have the open sets all the open sets from $X_f$, $X_g$ and $X$ and $X_f^{op},X_g^{op}$ and $X$ respectively. Thus, the identity maps $\mathbf{1}_X:X_{fg}\to X$ and $\mathbf{1}_X:X_{fg}^{op}\to X$ are continuous. We then we have the following corollary to \cref{theorem:stability_for_direct_images}.

\begin{corollary}
\label{corollary:modified_stability_1}
Let $X$ be a topological space and let $f,g:X\to \mathbb{R}^n$. Let $F$ be in $D^b(\cat{Sh}(X;\cat{k}))$. Then 
\[d_C(Rf_*\mathbf{1}_X^{-1}F,Rg_*\mathbf{1}_X^{-1}F)\le ||f-g||_{\infty}.\]
Dually, if $G$ is in $D^b(\cat{CoSh}(X;\cat{k}))$. Then 
\[d_C(Lf_{\dagger}\mathbf{1}_X^{-1}G,Lg_{\dagger}\mathbf{1}_X^{-1}G)\le ||f-g||_{\infty}.\]
\end{corollary}

\begin{proof}
Apply \cref{theorem:stability_for_direct_images} to $\mathbf{1}_{X}^{-1}F$ and $\mathbf{1}_{X}^{-1}G$.
\end{proof}

\subsection{Stability for modified direct images on the codomain}
Consider the following topology on $\mathbb{R}^n$. A set $U$ is open if and only if it is an up-set and also open in the Euclidean topology on $\mathbb{R}^n$. This is an example of a $\gamma$-topology on $\mathbb{R}^n$ where $\gamma$ is the cone $U_0$ \cite[Chapter 3.5.]{kashiwara1990sheaves}. Dually, we can declare a set $D$ to be open if and only if it is a down-set and open in the Euclidean topology. This is an example of a $\gamma^{op}$-topology where $\gamma^{op}$ is the antipodal cone to $\gamma$, $D_0$. Denote by $\mathbb{R}^n_{\gamma}$ the set $\mathbb{R}^n$ with the $\gamma$ topology and by $\mathbb{R}^n_{\gamma^{op}}$ the set $\mathbb{R}^n$ witht he $\gamma^{op}$ topology. Note that by construction the identity maps $\mathbf{1}_{\mathbb{R}^n}:\mathbb{R}^n_{\le}\to \mathbb{R}^n_{\gamma}$ and $\mathbf{1}_{\mathbb{R}^n}:\mathbb{R}^n_{\le^{op}}\to \mathbb{R}^n_{\gamma^{op}}$ are continuous as the topologies on the domains are finer by construction. If $f:X\to \mathbb{R}^n_{\gamma}$ is continuous we have the following diagram of topological spaces.
\begin{figure}[H]
\centering
\begin{tikzcd}
X\arrow[rd,"f"]\arrow[d,"f"]&&& X\arrow[rd,"g"]\arrow[d,"g"]\\
\mathbb{R}^n_{\le}\arrow[r,"\mathbf{1}_{\mathbb{R}^n}"]&\mathbb{R}^n_{\gamma}&& \mathbb{R}^n_{\ge}\arrow[r,"\mathbf{1}_{\mathbb{R}^n}"]&\mathbb{R}^n_{\gamma^{op}}
\end{tikzcd}
\caption{If $F$ is a sheaf on $X$ and $f:X\to \mathbb{R}^n_{\gamma}$ is continuous, we can construct the sheaf $\mathbf{1}_{\mathbb{R}^n}^{-1}f_*F$  on $\mathbb{R}^n_{\le}$. Similarly, if $G$ is a cosheaf on $X$ and $g:X\to \mathbb{R}^n_{\gamma^{op}}$ is continuous, we can construct the cosheaf $\mathbf{1}_{\mathbb{R}^n}^{-1}g_{\dagger}G$ on $\mathbb{R}^n_{\ge}$.}
\label{fig:5}
\end{figure}

\begin{theorem}
\label{theorem:modified_stability}
Let $f,g:X\to \mathbb{R}^n$ be set-theoretic maps. Suppose that $f,g:X\to \mathbb{R}^n_{\gamma}$ are continuous. Let $F$ be in $D^b(\cat{Sh}(X;\mathbf{k}))$. Then 
\[d_C(\mathbf{1}_{\mathbb{R}^n}^{-1}Rf_*F,\mathbf{1}_{\mathbb{R}^n}^{-1}Rg_*F)\le ||f-g||_{\infty}.\]
Dually, suppose that $f,g:X\to \mathbb{R}^n_{\gamma^{op}}$ are continuous. Let $G$ be in $D^b(\cat{CoSh}(X;\cat{k}))$. Then
\[d_C(\mathbf{1}_{\mathbb{R}^n}^{-1}Lf_{\dagger}G,\mathbf{1}_{\mathbb{R}^n}^{-1}Lg_{\dagger}G)\le ||f-g||_{\infty}.\] 
\end{theorem}

\begin{proof}
The proof is analogous to the proof of \cref{theorem:stability_for_direct_images}, we sketch out the details for the sheaf case. Suppose that $||f-g||_{\infty}=\epsilon$. Note that if $U\subseteq \mathbb{R}^n$ is open in the $\gamma$-topology, then $U(a)$ is also open in the $\gamma$-topology for all $a\in \mathbb{R}^n$. Suppose that $F\to E$ is an injective resolution of $F$. We thus need to show that $d_C(\mathbf{1}_{\mathbb{R}^n}^{-1}f_*E,\mathbf{1}_{\mathbb{R}^n}^{-1}g_*E)\le \epsilon$. We can construct chain maps $\Phi$ and $\Psi$ as in the proof of \cref{theorem:stability_for_direct_images}, on the complexes $f_*E$ and $g_*E$. In particular, for $U$ open in $\mathbb{R}^n_{\gamma}$, $\Phi^n(U)$ is the sheaf restriction map $E^n(f^{-1}(U(-\bm{\epsilon})))\to  E^n(g^{-1}(U))$ and $\Psi^n(U)$ is the sheaf restriction map $E^n(g^{-1}(U(-\bm{\epsilon})))\to  E^n(f^{-1}(U))$. Note that open set $U$ in $\mathbb{R}^n_{\gamma}$ is also open in $\mathbb{R}^n_{\le}$ by construction. Thus $\mathbf{1}_{\mathbb{R}^n}^{-1}f_*E(U)=f_*E(U)$. However the principal up-sets $U_x$ for $x\in \mathbb{R}^n$ are not open in $\mathbb{R}^n_{\gamma}$. In particular,  by definition we have that $\mathbf{1}_{\mathbb{R}^n}^{-1}f_*E(U_x)$ is the complex of sheaves associated to the complex of presheaves $\colim_{U_x\subseteq V}f_*E(V)$. Since we already constructed the maps $\Phi$ and $\Psi$, applying the colimit functor gives us extensions of $\Phi$ and $\Psi$ to $\mathbf{1}_{\mathbb{R}^n}^{-1}\Phi: \mathbf{1}_{\mathbb{R}^n}^{-1}*E(-\bm{\epsilon})\to \mathbf{1}_{\mathbb{R}^n}^{-1}g_*E$ and $\mathbf{1}_{\mathbb{R}^n}^{-1}\Psi:\mathbf{1}_{\mathbb{R}^n}^{-1}g_*E(-\bm{\epsilon})\to \mathbf{1}_{\mathbb{R}^n}^{-1}f_*E$. 
Thus, the functor $\mathbf{1}_{\mathbb{R}^n}^{-1}$ then gives us chain maps $\mathbf{1}_{\mathbb{R}^n}^{-1}(\Phi)$ and $\mathbf{1}_{\mathbb{R}^n}^{-1}(\Psi)$ whose localizations give us an $\epsilon$-isomorphism between $\mathbf{1}_{\mathbb{R}^n}^{-1}f_*E$ and $\mathbf{1}_{\mathbb{R}^n}^{-1}g_*E$ by construction.
\end{proof}

This approach, of modifying the topology on the codomain, also fixes the issue in \cref{example:sheaf_direct_image_circle} making the projection maps $p:S^1\to \mathbb{R}^n_{\gamma}$ and $p:S^1\to \mathbb{R}^n_{\gamma^{op}}$ continuous. More generally, if $M$ is a manifold and a map $f: M\to \mathbb{R}^n$ is continuous (with the Euclidean topology on $\mathbb{R}^n$) or even Morse, then the maps $f:M\to \mathbb{R}^n_{\gamma}$ and $f:M\to \mathbb{R}^n_{\gamma^{op}}$ are also continuous as the Euclidean topology on $\mathbb{R}^n$ is finer than both $\mathbb{R}^n_{\gamma}$ and $\mathbb{R}^n_{\gamma^{op}}$ by construction.  Thus, a lot of examples we might care about where a particular map $f:X\to \mathbb{R}^n_{\le}$ is not necessarily continuous, the same map on the modified codomain $\mathbb{R}^n_{\gamma}$ or $\mathbb{R}^n_{\gamma^{op}}$ will be continuous.

\section{Concluding remarks}
\label{section:concluding_remarks}

In this work, we have introduced two convolution operations on the derived category of functors $\cat{P}\to \cat{Vect}_{\cat{k}}$ using the sheaf and cosheaf perspectives. Our convolution operations can also be thought of as thickening of sheaves, originally introduced in the work by Curry in \cite{curry2013sheaves} and later expanded in the derived setting by Schapira and Kashiwara in \cite{kashiwara2018persistent}. After all in the case $P=\mathbb{R}^n$, for an up-set $U$, $U(-\bm{\epsilon})$ is the  up-set $U$ thickened by $\epsilon$ in the Hausdorff distance sense. 

We defined a convolution distance and showed stability theorems for direct images of sheaves and cosheaves on $\mathbb{R}^n$ with the Alexandrov topology. Currently, we have no way of computing this distance for arbitrary complexes of persistence modules. Perhaps looking at a particular subclass of complexes the distance is computable. It might also be possible to define a matching type distance that is isometric and easier to compute analogous to the work of Berbouk and Ginot in \cite{berkouk2018derived}.

Convolution operations for sheaves of vector spaces on $\mathbb{R}^n_{\gamma}$ and cosheaves of vector spaces on $\mathbb{R}^n_{\gamma^{op}}$ seem plausible as well. It does not seem that sheaves and cosheaves on these topological spaces are isomorphic to graded modules over a graded ring, thus a graded module tensor product is not defined. However one can still define convolutions by setting $M\bullet^LN:=Ls_{\dagger}(M\boxtimes N)$ and $M*^RN:=Rs_{*}(M\boxtimes N)$. The reason why $s_!$ might not be a good candidate is that we run into the same issue of non-Hausdorffness and thus cannot apply the Beck-Chevaley Theorem (\cref{theorem:beck_chevaley}) to do computations. Thus we suspect direct images are the correct choice rather than direct images with proper supports. 

\subsection*{Acknowledgments}
We thank Nicolas Berbouk for his feedback and pointing us to relevant literature that was missed in an earlier version of this work. In particular, we were not originally aware of the history of the functor $\scHom^{\star}$ from which we drew inspiration to define right derived functor $R\scHom^{*}$. This material is based upon work supported by, or in part by, the Army Research Laboratory and the Army Research Office under contract/grant number W911NF-18-1-0307.

\appendix

\section{Sheaves}
\label{section:background_sheaves}

We introduce some notions from sheaf theory. For more details see \cite[Chapter 1]{Bredon:SheafTheory} and
\cite[Chapter 2]{kashiwara1990sheaves}. Throughout this section, $X$ is a topological space. 

Given a presheaf $F$ on $X$ there exists a sheaf $F^+$ and a morphism $\theta:F\to F^+$ such that for any sheaf $G$ the homomorphism given by $\theta$:
 \[\Hom_{\mathbf{Sh}(X)}(F^+,G)\to \Hom_{\mathbf{PSh}(X)}(F,G)\]
 is an isomorphism. In other words, $F\mapsto F^+$ is the left adjoint functor of the inclusion functor $\cat{Sh}(X)\to \cat{PSh}(X)$. Moreover, $(F^+,\theta)$ is unique up to isomorphism, and for any $x\in X$, $\theta_x:F_x\to F^+_x$ is an isomorphism. The sheaf $F^+$ is called \emph{the sheaf associated to $F$} or \emph{sheafification of $F$.}

Given an abelian group $A$, we denote by $A_X$ the sheaf associated to the presheaf $U\mapsto A$, where $U$ is open in $X$, and we say $A_X$ is the \emph{constant sheaf} on $X$ with stalk $A$.
%

Let $\mathcal{R}$ be a sheaf of rings on $X$. The pair $(X,\mathcal{R})$ is called a \emph{ringed space}. A left $\mathcal{R}$-module $M$ is a sheaf of abelian groups $M$ such that for every open $U\subset X$, $M(U)$ is a left $\mathcal{R}(U)$-module, and for any inclusion $V\subset U$, $V$ and $U$ open, the restriction morphism is compatible with the structure of the module, that is, $M(V\subset U)(sm)=\mathcal{R}(V\subset U)(s)\cdot M(V\subset U)(m)$ for every $s\in \mathcal{R}(U)$ and $m\in M(U)$. Define right $\mathcal{R}$-modules in the obvious way and morphisms between left(right) modules is a natural transformation compatible with the structure of the module. Denote these sets of natural transformations by $\text{Hom}_{\mathcal{R}}(M,N)$. We denote the category of right $\mathcal{R}$-modules by $\shModR{\mathcal{R}}$, and the category of left $\mathcal{R}$-modules by $\shRMod{\mathcal{R}}$.

  Denote by $\mathbb{Z}_X$ the sheaf associated to the constant presheaf $U\mapsto \mathbb{Z}$ for every open $U\subset X$. Then $\mathbb{Z}_{X}$-modules are precisely sheaves with values in abelian groups, i.e, $\mathbf{Mod}(\mathbb{Z}_X)=\mathbf{Sh}(X)$. More generally, define $R_X$ to be the sheaf associated to the constant presheaf $U\mapsto R$ for every open $U\subset X$.
  For example, we have the constant sheaf $\mathbf{k}_{\R^n}$.


Let $F$ be a right $\mathcal{R}$-module and $G$ be a left $\mathcal{R}$-module. Define $F\otimes_{\mathcal{R}}G$ to be the sheaf associated to the presheaf of abelian groups $U\mapsto F(U)\otimes_{\mathcal{R}(U)}G(U)$, and call $F\otimes_{\mathcal{R}}G$ the tensor product of $F$ and $G$ over $\mathcal{R}$.

\begin{definition}{\cite[Definition 2.2.11]{kashiwara1990sheaves}}
\label{def:locally_constant_sheaf}
\begin{itemize}
\item[1)] Let $A$ be an abelian group. One denotes by $A_X$ the sheaf associated to the presheaf $U\mapsto A$, for $U\subseteq X$ open, and says that $A_X$ is the \emph{constant sheaf on $X$ with stalk $A$}.
\item[2)] Let $F$ be a sheaf on $X$. One says $F$ is \emph{locally constant on} $X$ if there exists an open covering $X=\bigcup\limits_{i} U_i$ such that for each $i$, $F|_{U_i}$ is a constant sheaf.  
\end{itemize}
\end{definition}

\begin{definition}{\cite[Notation 2.3.12]{kashiwara1990sheaves}}
\label{def:external_tensor_product}
Let $p_X:X\to S$ and $p_Y:Y\to S$ be two continuous maps, and let $X\times_S Y$ be the fiber product of $X$ and $Y$ over $S$. Denote by $q_1$ and $q_2$ the projections from $X\times_SY$ to $X$ and $Y$ respectively, and by $p$ the projection $X\times_S Y\to S$. Let $\mathscr{R}$ be a sheaf of rings on $S$, let $F$ (respectively $G$) be a sheaf of $p_X^{-1}(\mathscr{R}^{op})$-modules (respectively $p_Y^{-1}\mathscr{R}$-modules). One sets:
\[F\underset{S}{\boxtimes_{\mathscr{R}}}G:=q_1^{-1}F\otimes_{p^{-1}\mathscr{R}}q_2^{-1}G\]
If there is no risk of confusion, we write $F\boxtimes_S G$, and if $S$ is the one point space, we simply write $F\boxtimes G$.
The sheaf $F\boxtimes G$ is called the \emph{external tensor product} of $F$ and $G$ (over $S$).
\end{definition}

\begin{definition}
\label{def:support_of_a_section}
Let $F$ be a sheaf on $X$. Define the \emph{support of a section} $s$ of $F$ on an open set $U$ as the complement in $U$ of the union of open sets $V\subseteq U$ such that the restriction of $s$ on $V$ is $0$. Denote this set by $\text{supp}(s)$. More explicitly, we have $\text{supp}(s)=\{x\in U\,|\, s_x=0\}$.
\end{definition}

\begin{definition}{\cite[Section 2.5]{kashiwara1990sheaves}}
\label{def:direct_image_with_compact_supports}
Let $f:Y\to X$ be continuous. Recall that $f$ is \emph{proper} if $f$ is closed and its fibers are relatively Hausdorff (two distinct points in the fiber have disjoint neighborhoods in $Y$) and compact. If $X$ and $Y$ are locally compact, $f$ is proper if and only if the inverse image of any compact subset of $X$ is compact. Let $G$ be a sheaf on $Y$. Let $f_!G$ be the subsheaf of $f_*G$ defined by:
\[\Gamma(U;f_!G):=\{s\in \Gamma(f^{-1}(U);G)\,|\, f:\text{supp}(s)\to U\, \text{is proper}\}\]
for all $U\subseteq X$ open. This sheaf is called the \emph{direct image with proper supports} of $G$.
We also define 
\[\Gamma_c(U;F)=\{s\in \Gamma(U;F)\,|\, \text{supp}(s)\, \text{is compact and Hausdorff}\}\]
\end{definition}

\begin{theorem}{\cite[Theorem 2.3.26]{dimca2004sheaves}}[Beck-Chevaley or Proper Base Change Theorem]
\label{theorem:beck_chevaley}
Let $X$ and $Y$ be locally compact Hausdorff spaces, $f:Y\to X$ a continuous map and $G$ a sheaf on $Y$. Then for all $x\in X$, there is a canonical isomorphism
\[(f_!G)_x\to \Gamma_c(f^{-1}(x);G|_{f^{-1}(x)})\]
Moreover, we also have a canonical isomorphism
\[(R^nf_!G)_x\cong R^n\Gamma_c(f^{-1}_x;G|_{f^{-1}(x)})\]
\end{theorem}


\printbibliography
\end{document}